\documentclass[11pt]{amsart}
\usepackage{amssymb,latexsym}
\newtheorem{theorem}{Theorem}[section]
\newtheorem{prop}[theorem]{Proposition}
\newtheorem{lemma}[theorem]{Lemma}
\newtheorem{remark}[theorem]{Remark}

\newtheorem{definition}[theorem]{Definition}
\newtheorem{cor}[theorem]{Corollary}
\newtheorem{example}[theorem]{Example}
\newcommand{\Z}{{\mathbb{Z}}}

\begin{document}

\title{Positive divisors in symplectic geometry}

\author{Jianxun Hu %\& Zhenbo Qin
                  \& Yongbin Ruan}
\address{Department of Mathematics\\ Zhongshan
University\\Guangzhou\\ P. R. China}
\email{stsjxhu@mail.sysu.edu.cn}
\thanks{${}^1$Partially supported by the NSFC Grant 10631050, NKBRPC(2006CB805905) and NCET-04-0795}
\address{Department of Mathematics\\ University of Michigan\\
Ann Arbor, MI 48109-1109\\  and \\ Yangtze Center of Mathematics
\\ Sichuan University \\ Chengdu, 610064, P.R.
China}\email{ruan@umich.edu}
\thanks{${}^3$supported by NSF Grant}

\maketitle

\tableofcontents

\section{Introduction}

Divisors or codimension two symplectic submanifolds play an
important role in Gromov-Witten theory and symplectic geometry.
For example, there is now a well-known  degeneration operation to
decompose a symplectic manifold $X$ into $X_1\cup_Z X_2$ a union
of $X_1, X_2$, along a common divisor $Z$. To utilize the above
degeneration to compute Gromov-Witten invariants inductively, one
needs to develop the relative Gromov-Witten theory of the pair
$(X,Z)$. Such a theory and its degeneration formula were first
constructed by Li-Ruan \cite{LR} ( see Ionel and Parker \cite{IP}
for a different version and  J. Li \cite{Li1, Li2} for an
algebraic treatment).

Maulik and Pandharipande \cite{MP} systematically studied the
degeneration formula for the degeneration $X\rightarrow X\cup_Z
P_Z$, where $P_Z$ is the projective closure of the normal bundle
$N_Z$. By introducing a certain partial order on relative
invariants, they interpreted the degeneration formula as a
``correspondence", a complicated upper triangular linear map from
relative invariants to absolute invariants. One consequence of
their correspondence is that the absolute and relative invariants
determine each other. Such a correspondence is a very powerful
tool in determining the totality of Gromov-Witten theory. But it
is not effective in determining any single invariant. A natural
question is if a divisor with a stronger condition will give us a
much stronger correspondence? We answer this affirmatively for
so-called positive symplectic divisors. We call a symplectic
divisor $Z$ {\em positive} if for some tamed almost complex
structure $J$, $C_1(N_Z)(A)>0$ for any $A$ represented by a
non-trivial $J$-sphere in $Z$. This is a generalization of ample
divisor from algebraic geometry. One of our main theorems is the
following comparison theorem.

We call $0\neq A\in H_2(Z, \Z)$ {\em stably effective} if there is
a nonzero genus zero Gromov-Witten invariant of $Z$ with class
$A$. Let $\pi: P_Z\rightarrow Z$ be the projection. Maulik and
Pandharipande show that any nonzero absolute or relative
Gromov-Witten invariant of $P_Z$ is determined by the
Gromov-Witten invariants of $Z$. It is a direct consequence of
their argument that if a nonzero genus zero Gromov-Witten or
relative Gromov-Witten invariant of $P_Z$ has class $B$, then
$\pi_*(B)=\sum_i a_i A_i$ for stably effective  $A_i\in H_2(Z,
\Z)$ and $a_i\in \Z_+$. Let $V=\min\, \{ C_1(N_Z(A))\mid A\in
H_2(Z,\Z)\, \mbox{is stably effective}\}$.

\begin{theorem}
Suppose that $Z$ is a positive divisor and  $V\geq  l$.  Then for
$A\in H_2(X,{\mathbb Z})$, $\alpha_i\in H^*(X,{\mathbb R})$, $
1\leq i\leq \mu$, and $\beta_j\in H^*(Z,{\mathbb R})$, $1\leq
j\leq l$, we have
\begin{eqnarray*}
&&\langle \alpha_1, \cdots,  \alpha_\mu,
 \iota^!(\beta_1), \cdots,  \iota^!(\beta_l)
 \rangle ^X_A\\
&=&  \sum_{\mathcal T}\langle  \alpha_1,\cdots, \alpha_\mu\mid
{\mathcal T}\rangle_A^{X,Z},
\end{eqnarray*}
where $\iota : Z\hookrightarrow X$  and $\iota^! = PD\iota_*PD$,
the summation runs over all possible weighted partitions
${\mathcal T}= \{(1, \gamma_1)$, $\cdots, (1,\gamma_q),
(1,[Z]),\cdots, (1,[Z])\}$ of $Z\cdot A$, where $\gamma_i$'s are
the products of some $\beta_j$ classes.
\end{theorem}

McDuff \cite{M1} also considered the similar comparison result in
the some special case.  For readers familiar with
Maulik-Pandharipande's relative/absolute correspondence, the above
theorem means that there is no lower order term in the
degeneration formula.

 The second motivation comes from symplectic birational geometry.
A fundamental problem in symplectic geometry is to generalize
birational geometry to symplectic geometry. In the 80's, Mori
introduced a program towards the birational classification of
algebraic manifolds of dimension three and up. In 90's, the last
author \cite{R1} speculated that there should be a symplectic
geometric program.  First of all, since there is no notion of
regular or rational function in symplectic geometry, we must
therefore first define the appropriate notion of symplectic
birational equivalence. For this purpose, in \cite{HLR}, the
authors proposed to use Guillemin-Sternberg's birational cobordism
to replace birational maps. Secondly, we need to study what
geometric properties behave nicely under this birational
cobordism. In \cite{HLR}, the authors defined the uniruledness of
symplecitc manifolds by requiring a nonzero Gromov-Witten
invariant with a point insertion and settled successfully the
fundamental birational cobordism invariance of uniruledness.
McDuff \cite{M2} proved that Hamiltonian $S^1$-manifolds are
uniruled using the techniques from \cite{HLR}.

Furthermore, Tian-Jun Li and the second author \cite{LtjR}
investigate the dichotomy of uniruled symplectic divisors. The
dichotomy asserts that if the normal bundle is non-negative in
some sense, then the ambient manifold is uniruled.

It is clear that our stronger comparison theorem should yield
stronger results. This is indeed the case. As an application of
our comparison theorem, we investigate symplectic rationally
connected manifolds.  Similar to the case of uniruledness (see
\cite{HLR}), we define the notion of $k$-point (strongly) rational
connectedness by requiring a non-zero (primary) Gromov-Witten
invariant with $k$ point insertions (see section five for the
detailed discussion). Of course, one important problem here of
interest is whether the notion of symplectic $k$-point rational
connectedness is invariant under birational cobordism given in
\cite{HLR}. From the blowup formula of \cite{H1, H2, H3, HZ, La},
we know that symplectic rational connectedness is invariant under
the symplectic blowup along points and some special submanifolds
with convex normal bundles. The general case is still unknown. We
should mention that a longstanding problem in Gromov-Witten theory
is to characterize algebraic rationally connectedness in terms of
Gromov-Witten theory. Our second main theorem is the following
theorem analogous to the theorem of McDuff \cite{M3, LtjR} for
uniruled divisors.

\begin{theorem}\label{rcsd-intro}
Let $(X,\omega)$ be a compact $2n$ dimensional symplectic manifold
which contains a submanifold $P$ symplectomorphic to ${\mathbb
P}^{n-1}$ whose normal Chern number  $m\geq 2$, then $X$ is
strongly rationally connected.

\end{theorem}

The paper is organized as follows. In section two, we first review
Gromov-Witten theory and its degeneration formula to set up the
notation. In Section three, we prove some vanishing and
non-vanishing results for relative Gromov-Witten invariants of
${\mathbb P}^1$-bundles. In section four, we prove a comparison
theorem between absolute and relative Gromov-Witten invariants. In
section five, we generalize the {\it from divisor to ambient
space} inductive construction of \cite{LtjR} to the case $k$-point
rational connectedness.

\bigskip\noindent
{\bf Acknowledgements:} Both of the authors would like to thank
Zhenbo Qin who took part in the early time of this work. This
article was prepared during the first author's visit to the
Department of Mathematics, University of Michigan-Ann Arbor and
the Department of Mathematics, University of Missouri-Columbia.
The first author is grateful to the departments for their
hospitality. The authors would like to thank Tian-Jun Li for his
valuable suggestions and telling us about McDuff's work. Thanks
also to A. Greenspoon for his editorial assistance.

\section{Preliminaries}

In this section, we want to review briefly the constructions of
virtual integration in the definitions of the absolute and
relative GW invariants, which are the main tool of our paper. We
refer to \cite{R1,LR} for the details.

\subsection{GW-invariants}\label{gw}

Suppose that $(X,\omega)$ is a compact symplectic manifold  and
$J$ is a tamed almost complex structure.

\begin{definition} \label{stable map}A stable $J-$holomorphic map
is an equivalence class of pairs $(\Sigma, f)$. Here $\Sigma$ is a
connected nodal marked Riemann surface with arithmetic genus $g$,
$k$ smooth marked points $x_1, ..., x_k$, and $f : \Sigma
\longrightarrow X$ is a continuous map whose restriction to each
component of $\Sigma$ (called a component of $f$ in short) is
$J$-holomorphic. Furthermore, it satisfies the stability
condition: if $f|_{S^2}$ is constant (called a ghost bubble) for
some $S^2-$component, then the  $S^2-$component has at least three
special points (marked points or nodal points). $(\Sigma, f)$,
$(\Sigma', f')$ are equivalent, or $(\Sigma,f) \sim (\Sigma',
f')$, if there is a biholomorphic map $h : \Sigma' \longrightarrow
\Sigma$ such that $f'=f\circ h$.
\end{definition}

An essential feature of Definition \ref{stable map} is that, for a
stable $J-$holomorphic map $(\Sigma, f)$, the automorphism  group
$$
   \mbox{Aut}(\Sigma, f) = \{ h\mid h\circ (\Sigma, f) = (\Sigma, f)\}
$$
is finite. We define the moduli space $\overline{\mathcal
M}^X_A(g,k, J)$ to be the set of equivalence classes of stable
$J-$holomorphic maps such that $[f] = f_*[\Sigma] =A\in H_2(X,{\bf
Z})$. The virtual dimension of $\overline{\mathcal M}^X_A(g,k,J)$
is computed by index theory,
$$
\mbox{virdim}_{\mathbb R}\overline{\mathcal M}^X_A(g,k,J) =
2c_1(A) + 2(n-3)(1-g) +2k,
$$
where $n$ is the complex dimension of $X$.

Unfortunately, $\overline{\mathcal M}^X_A(g,k,J)$ is highly
singular and may have larger dimension than the virtual dimension.
To extract invariants, we use the following virtual neighborhood
method.

First, we drop the $J$-holomorphic condition from the previous
definition and require only that each component of $f$ be smooth.
We call the resulting object a stable map or a $C^\infty$-stable
map. Denote the corresponding space of equivalence classes by
$\overline{\mathcal B}^X_A(g,k,J)$. $\overline{\mathcal
B}^X_A(g,k,J)$ is clearly an infinite dimensional space. It has a
natural stratification given by the topological type of $\Sigma$
together with the fundamental classes of the components of  $f$.
The stability condition ensures that $\overline{\mathcal
B}^X_A(g,k,J)$ has only finitely many strata such that each
stratum is a Fr\'echet orbifold. Further, one can use the
pregluing construction to define a topology on $\overline{\mathcal
B}^X_A(g,k,J)$ which is Hausdorff and makes $\overline{\mathcal
M}^X_A(g,k,J)$ a compact subspace. (see \cite{R1}).

We can define another infinite dimensional space $\Omega^{0,1}$
together with a map
$$
  \pi : \Omega ^{0,1} \longrightarrow \overline{\mathcal
B}^X_A(g,k,J)
$$
such that the fiber is $\pi^{-1}(\Sigma,f) = \Omega ^{0,1}
(f^*TX)$. The Cauchy-Riemann operator is now interpreted as a
section of $ \pi : \Omega ^{0,1} \longrightarrow
\overline{\mathcal B}^X_A(g,k,J)$,
$$\overline{\partial}_J:\overline{\mathcal B}^X_A(g,k,J)\to \Omega ^{0,1}$$
with $\overline{\partial}_J^{-1}(0)$  nothing but
$\overline{\mathcal M}^X_A(g,k,J)$.

 At each $(\Sigma,f) \in
\overline{\mathcal M}^X_A(g,k,J)$, there is a canonical
decomposition of the tangent space of $\Omega^{0,1}$ into the
horizontal piece and the vertical piece. Thus we can linearize
$\overline{\partial}_J$ with respect to deformations of stable
maps and project to the vertical piece to obtain an elliptic
complex
\begin{equation}\label{virtual-0}
    L_{\Sigma,f}: \Omega^0(f^*TX) \longrightarrow
    \Omega^{0,1}(f^*TX).
\end{equation}

Several explanations are in order  for  formula
$(\ref{virtual-0})$. Choose a compatible Riemannian  metric on $X$
and let $\nabla$ be the Levi-Civita connection.

When $\Sigma$ is irreducible, $\nabla$ induces a connection on
$f^*TX$, still denoted by $\nabla$. Then $L_{\Sigma,f} =
\overline{\nabla}$, where $\overline{\nabla}$ is the projection of
$\nabla$ onto the $(0,1)$-factor.

When $\Sigma$ is reducible, formula (\ref{virtual-0}) is
interpreted as follows. For simplicity, suppose that $\Sigma$ is
the union of $\Sigma_1$ and  $\Sigma_2$ intersecting at $p\in
\Sigma_1$ and  $q\in \Sigma_2$. Let the corresponding maps be
$f_1$ and  $f_2$. Then define
$$
\Omega^0(f^*TX) = \{(v_1,v_2)\in \Omega^0(f^*_1TX) \times
\Omega^0(f^*_2TX) \mid v_1(p) = v_2(q)\}
$$
and
$$
\Omega^{0,1}(f^*TX) = \Omega^{0,1}(f^*_1TX)\oplus
\Omega^{0,1}(f^*_2TX).
$$
$L_{\Sigma,f}$ is then the restriction of $L_{\Sigma_1,f_1}\oplus
L_{\Sigma_2, f_2}$ to $\Omega^0$. It is clear that
$$
\mbox{Ker} L_{\Sigma,f} = \{ (v_1,v_2)\in \mbox{Ker}\,
L_{\Sigma_1, f_1} \times \mbox{Ker}\, L_{\Sigma_2, f_2}\mid v_1(p)
= v_2(q)\}.
$$
To understand $\mbox{Coker}\, L_{\Sigma,f}$ geometrically, it is
convenient to use another elliptic complex. The idea is as
follows. We would like to drop the condition $v_1(p) = v_2(q)$. To
keep the
 index unchanged, we need to enlarge $\Omega^{0,1}(f^*TX)$. A standard method
motivated by algebraic geometry is to allow a simple pole at the
intersection point. This leads to
$$
 \tilde{L}_{\Sigma,f}:
\tilde{\Omega}^0(f^*TX)\longrightarrow
\tilde{\Omega}^{0,1}(f^*TX),
$$
where
\begin{eqnarray}
  \tilde{\Omega}^0 (f^*TX)&= &\Omega^0(f_1^*TX) \times \Omega^0(f_2^*TX),\nonumber \\
\tilde{\Omega}^{0,1}(f^*TX) & = & \{(v_1,v_2)\in
\Omega^{0,1}(f_1^*TX\otimes {\mathcal O}(p))\times
\Omega^{0,1}(f_1^*TX\otimes {\mathcal O}(q)) \mid \nonumber \\
 & & \mbox{Res}_p v_1 + \mbox{Res}_q v_2 =0\}\nonumber,
\end{eqnarray}
and
$$
  \tilde{L}_{\Sigma,f} = \tilde{L}_{\Sigma_1, f_1} \oplus
  \tilde{L}_{\Sigma_2,f_2}.
$$
It is well-known that
$$
  \mbox{Ker}\, L_{\Sigma,f} \cong \mbox{Ker}\, \tilde{L}_{\Sigma,f}, \hspace{1 cm}
  \mbox{Coker}\, L_{\Sigma,f} \cong \mbox{Coker}\, \tilde{L}_{\Sigma, f}.
$$
Therefore
\begin{eqnarray}
\mbox{Coker}\, L_{\Sigma,f} = & &\{ (v_1,v_2)\in \mbox{Coker}\,
\tilde{L}_{\Sigma_1,f_1} \times \mbox{Coker}\,
\tilde{L}_{\Sigma_2,f_2} \mid \nonumber \\
& &\mbox{Res}_p v_1 + \mbox{Res}_q v_2 = 0\}.\nonumber
\end{eqnarray}

If $\Sigma$ has more than two components, the construction above
extends in a straightforward fashion.

To consider the full linearization of $\overline{\partial}_J$, we
have to include the deformation space $Def(\Sigma)$ of the nodal
marked Riemann surface $\Sigma$. $Def(\Sigma)$ fits into  the
short exact sequence,
$$
  0 \longrightarrow H^1(T\Sigma) \longrightarrow Def(\Sigma)
  \longrightarrow T_p\Sigma_1\otimes T_q\Sigma_2 \longrightarrow
  0,
$$
where the first term represents the deformation space of $\Sigma$
preserving the nodal point and the third term represents the
smoothing of the nodal point. Moreover, $H^1(T\Sigma)$ is a
product, with each factor being the deformation space of a
component while treating  the nodal point as a new marked point.

The full linearization of $\overline{\partial}_J$ is given by
$L_{\Sigma,f}\oplus \frac{1}{2}Jdf$. Denote
\begin{eqnarray}
 Def(\Sigma,f) &&= \mbox{Ker} (L_{\Sigma,f}\oplus \frac{1}{2}
Jdf)/T_e\mbox{Aut}(\Sigma),\nonumber\\
Obs(\Sigma, f) &&= \mbox{Coker} (L_{\Sigma, f} \oplus
\frac{1}{2}Jdf)\nonumber.
\end{eqnarray}
When $Obs(\Sigma, f) = 0$, $(\Sigma,f)$ is a smooth point of the
moduli space and $Def(\Sigma,f)$ is its tangent space.

 Now we choose a nearby symplectic form $\omega'$ such that
$\omega'$ is tamed with $J$ and $[\omega']$ is a rational
cohomology class. Using $\omega'$, B. Siebert \cite{S1} (see also
the appendix in \cite{R1}) constructed a natural finite
dimensional vector bundle over $\overline{\mathcal B}_A(g,k,J)$.
It has  the property of dominating any local finite dimensional
orbifold bundle as follows. Let $U$ be a neighborhood  of
$(\Sigma, f)\in \overline{\mathcal B}_A(g,k,J)$ and
 $F_U$ be an orbifold bundle over $U$. Then Siebert
constructed a bundle $E$ over $\overline{\mathcal B}_A(g,k,J)$
such that there is a surjective bundle map $E|_U\longrightarrow
F_U$.

For each $(\Sigma, f)\in \overline{\mathcal M}_A(g,k,J)$,
$Obs(\Sigma,f)$ extends to a local orbifold bundle $F(\Sigma,f)$
over a neighborhood $\tilde{\mathcal U}_{\Sigma,f}$ of
$(\Sigma,f)$. Then we use Siebert's construction to find a global
orbifold bundle ${\mathcal E}(\Sigma,f)$ dominating $F(\Sigma,f)$.
In fact, any global orbifold bundle with this property will work.
We also remark that it is often convenient to replace
$Obs(\Sigma,f)$ by $\mbox{Coker}\, L_{\Sigma,f}$ in the
construction. Over each $\tilde{\mathcal U}_{\Sigma, f}$, by the
domination property, we can construct a stabilizing term
$\eta_{\Sigma,f}: {\mathcal E}(\Sigma,f)\longrightarrow
\Omega^{0,1}$ supported in $\tilde{\mathcal U}_{\Sigma,f}$ such
that $\eta_{\Sigma,f}$ is surjective onto $Obs(\Sigma,f)$ at
$(\Sigma,f)$. Obviously, $\eta_{\Sigma,f}$ can be viewed as a map
from ${\mathcal E}(\Sigma,f)$ to $\Omega^{0,1}$. Then the
stabilizing equation
$$\overline{\partial}_J + \eta_{\Sigma,f}:{\mathcal E}(\Sigma, f)\to \Omega^{0,1},
\quad ((\Sigma', f'), e)\to \overline{\partial}_J
f'+\eta_{\Sigma,f}(e)$$ has no cokernel at $(\Sigma,f)$. By
semicontinuity, it has no cokernel in a neighborhood ${\mathcal
U}_{\Sigma,f} \subset \tilde{\mathcal U}_{\Sigma,f}$.

Since $\overline{\mathcal M}^X_A(g,k,J)$ is compact, there are
finitely many ${\mathcal U}_{\gamma} = {\mathcal
U}_{\Sigma_{\gamma},f_{\gamma}}$ covering $\overline{\mathcal
M}^X_A(g,k,J)$. Let
$$
 {\mathcal U} = \cup_{\gamma} {\mathcal U}_{\gamma}, \hspace{0.5 cm} {\mathcal E} =\oplus_{\gamma}
 {\mathcal E}(\Sigma_{\gamma},f_{\gamma}), \hspace{0.5 cm} \eta = \sum_{\gamma}
 \eta_{\Sigma_{\gamma},f_{\gamma}}.
$$
Consider  the finite dimensional vector bundle  over ${\mathcal
U}$, $p: {\mathcal E}|_{\mathcal U} \longrightarrow {\mathcal U}$.
The stabilizing equation $\overline{\partial}_J + \eta$ can be
interpreted as a section of the bundle $p^*\Omega^{0,1}\to
{\mathcal E}|_{\mathcal U}$. By construction, this section
$$
\overline{\partial}_J + \eta: {\mathcal E}|_{\mathcal U}\to
p^*\Omega^{0,1}$$ is transverse to the zero section of
$p^*\Omega^{0,1}\to {\mathcal E}|_{\mathcal U}$.

The set $U^X_{{\mathcal S}_e} = (\overline{\partial}_J
+\eta)^{-1}(0)$    is called the virtual neighborhood in [R1]. The
heart of \cite{R1} is to show that $U^X_{{\mathcal S}_e} $ has the
structure of a $C^1-$manifold.

Notice that $U^X_{{\mathcal S}_e} \subset {\mathcal E}|_{\mathcal
U}$.
 Over  $U^X_{{\mathcal S}_e}$ there is the tautological bundle
 $${\mathcal E}_X =
p^*({\mathcal E}|_U)|_{U^X_{{\mathcal S}_e}}. $$ It comes with the
tautological inclusion map
$$
  S_X: U^X_{{\mathcal S}_e} \longrightarrow {\mathcal E}_X, \quad ((\Sigma', f'), e)\to e,
$$
which can be viewed as a section of ${\mathcal E}_X$. It is easy
to check that
$$S^{-1}_X(0) = \overline{\mathcal M}^X_A(g,k,J).$$
Furthermore, one can show that $S_X$ is a proper section.

Note that the stratification of $\overline{\mathcal B}_A(g,k,J)$
induces a natural stratification of ${\mathcal E}$. We can define
$\eta_{\gamma} = \eta_{\Sigma_{\gamma},f_{\gamma}}$ inductively
from lower stratum to higher stratum. For example, we can first
define $\eta_{\gamma}$ on a stratum and extend to a neighborhood.
Then we define $\eta_{\gamma+1}$ at the next stratum supported
away from lower strata. One consequence of this construction is
that $U^X_{{\mathcal S}_e}$ has the same stratification as that of
$\mathcal E$. Namely, if ${\mathcal B}_{D'}\subset
\overline{\mathcal B}_D$ is a lower stratum,
$$
  U^X_{{\mathcal S}_e}\cap {\mathcal E}|_{{\mathcal B}_{D'}}\subset
  U^X_{{\mathcal S}_e}\cap{\mathcal E}|_{\overline{\mathcal B}_D}
$$
is a submanifold of codimension at least $2$.

There are evaluation maps
$$
  ev_i: \overline{\mathcal B}_A(g,k,J)\longrightarrow X, \quad
  (\Sigma, f) \to f(x_i),
$$
for $1\leq i\leq k$. $ev_i$ induces a natural map from
$U_{{\mathcal S}_e}\longrightarrow X^k$, which can be shown to be
smooth.

Let $\Theta$ be the Thom form of the finite dimensional bundle
${\mathcal E}_X\to U^X_{{\mathcal S}_e}$.

\begin{definition} The (primary) Gromov-Witten invariant is
defined as
$$
  \langle\alpha_1,\cdots,\alpha_k\rangle^X_{g,A} = \int_{U_{{\mathcal
  S}_e}}S_X^*\Theta\wedge\Pi_i ev_i^*\alpha_i,
$$
where $\alpha_i\in H^*(X;{\mathbb R})$. For the genus zero case,
we also write $\langle\alpha_1,\cdots,\alpha_k\rangle^X_A=\\
\langle\alpha_1,\cdots,\alpha_k\rangle^X_{0,A}$.
\end{definition}

\begin{definition} For each marked point $x_i$, we define an orbifold complex line
bundle ${\mathcal L}_i$ over $\overline{\mathcal B}^X_A(g,k,J)$
whose fiber is $T_{x_i}^*\Sigma$ at $(\Sigma,f)$. Such a line
bundle can be pulled back to $U^X_{{\mathcal S}_e}$ (still denoted
by ${\mathcal L}_i$). Denote $c_1({\mathcal L}_i)$, the first
Chern class of ${\mathcal L}_i$, by $\psi_i$.
\end{definition}

\begin{definition} The descendent Gromov-Witten invariant is
defined as
$$
  \langle\tau_{d_1}\alpha_1,\cdots, \tau_{d_k}\alpha_k\rangle^X_{g,A}  =
 \int_{U^X_{{\mathcal S}_e}}S_X^*\Theta\wedge \Pi_i\psi_i^{d_i}\wedge ev_i^*\alpha_i,
$$
where $\alpha_i\in H^*(X;{\mathbb R})$.
\end{definition}

\begin{remark}
In the stable range $2g +k \geq 3$, one can also define
non-primary GW invariants (See e.g. \cite{R1}). Recall that there
is a map $\pi: \overline{\mathcal B}^X_A(g,k,J)\rightarrow
\overline{\mathcal M}_{g,k}$ contracting the unstable components
of the source Riemann surface. We can introduce a class $\kappa$
from the Deligne-Mumford space via $\pi$ to define  the ancestor
GW invariants
$$
  \langle \kappa\mid \Pi_i\alpha_i\rangle^X_{g,A} =
 \int_{U^X_{{\mathcal S}_e}}S_X^*\Theta\wedge \pi^*\kappa\wedge \Pi_i ev_i^*\alpha_i.
$$
The primary Gromov-Witten invariants are the special invariants
with the point class in $\overline{\mathcal M}_{0,k}$.
\end{remark}

\begin{remark}\label{D} For computational purpose we would
mention the following variation of the virtual neighborhood
construction.
 Suppose $\iota: D\subset X$ is a submanifold. For $\alpha  \in H^*(D;
{\mathbb R})$ we define  $\iota^!(\alpha) \in H^*(X;{\mathbb R})$
via the transfer map $\iota^! = PD_X\circ \iota_* \circ PD_D$. One
can construct Gromov-Witten invariants with an insertion of the
form $\iota^!(\alpha)$ as follows. Apply the virtual neighborhood
construction to the compact subspace
$$\overline{\mathcal M}_A(g,k,J)\cap ev_1^{-1}(D)$$
in $ \overline{\mathcal B}^X_A(g,k,J,D) = ev_1^{-1}(D)$  to obtain
a virtual neighborhood ${\mathcal U}_{{\mathcal S}_e}(D)$ together
with the natural map $ev_D : {\mathcal U}_{{\mathcal S}_e}(D)
\longrightarrow D$. It is easy to show that
\begin{eqnarray*}
  \langle \tau_{d_1}\iota^!(\alpha), \tau_{d_2}\beta_2,\cdots,
  \tau_{d_k}\beta_k\rangle^X_{g,A}
    =  \int_{U_{{\mathcal
 S}_e}(D)}S^*\Theta\wedge ev_D^*\alpha\wedge \prod_{i=2}^k \psi_i^{d_i}ev_i^* \beta_i
  .
\end{eqnarray*}
\end{remark}

\begin{remark} \label{absolute graph}
        For each $ \langle\tau_{d_1}\alpha_1,\cdots, \tau_{d_k}\alpha_k\rangle^X_{g,A}$, we can conveniently associate a simple graph $\Gamma$ of one vertex decorated by
        $(g,A)$ and   a tail for each marked point. We then further decorate
        each tail by $(d_i, \alpha_i)$ and call the  resulting graph $\Gamma(\{(d_i,\alpha_i)\})$ {\em a weighted
        graph}. Using the weighted graph
        notation, we denote the above invariant by $\langle \Gamma(\{(d_i, \alpha_i)\})\rangle^X$.
        We can also consider the
       disjoint union $\Gamma^{\bullet}$ of several such graphs  and use $A_{\Gamma^{\bullet}},
       g_{\Gamma^{\bullet}}$ to denote the total homology class and total arithmetic genus.
       Here the total arithmetic genus is $1+\sum (g_i-1)$.  Then,
       we define $\langle \Gamma^{\bullet}(\{(d_i, \alpha_i)\})\rangle^X$ as the  product
       of Gromov-Witten invariants of the connected components.
       \end{remark}

\subsection{Relative GW-invariants}\label{rgw}

In this section, we will review the relative GW-invariants.
 The readers can
find more details in the reference \cite{LR}.

Let $Z\subset X$ be a real codimension $2$ symplectic submanifold.
Suppose that $J$ is an $\omega-$tamed almost complex structure on
$X$ preserving $TZ$, i.e. making $Z$ an almost complex
submanifold. The relative Gromov-Witten invariants are defined by
counting the number of stable $J-$holomorphic maps intersecting
$Z$ at finitely many points with prescribed tangency. More
precisely, fix a $k$-tuple $T_k=(t_1, \cdots, t_k)$ of positive
integers, consider a marked pre-stable curve
$$
(C,x_1, \cdots, x_m, y_1, \cdots, y_k)
$$
and stable $J-$holomorphic map $
 f: C\longrightarrow X
$ such that the divisor $f^*Z$ is
$$
f^*Z = \sum_i t_i y_i.
$$
One would like to consider the moduli space of such curves and
apply the virtual neighborhood technique to construct the relative
invariants. But this scheme needs modification as the moduli space
is not compact. It is true that for a sequence of $J-$holomorphic
maps $(\Sigma_n,f_n)$ as above, by possibly passing to a
subsequence, $f_n$ will still converge to a stable $J$-holomorphic
map $(\Sigma,f)$. However the limit $(\Sigma,f)$ may have some
$Z-$components, i.e. components whose images under $f$ lie
entirely in $Z$.

To deal with this problem  the authors in \cite{LR} adopt  the
open cylinder model.  Choose a Hamiltonian $S^1$ function $H$ in a
closed  $\epsilon-$symplectic tubular neighborhood $X_0$ of $Z$
with $H(X_0)=[-\epsilon, 0]$ and $Z=H^{-1}(-\epsilon)$. Next we
need
 to choose an almost complex structure with nice properties near $Z$.
An almost complex structure $J$ on $X$ is said to be tamed
relative to $Z$  if $J$ is $\omega$-tamed, $S^1-$invariant for
some $(X_0, H)$, and  such that $Z$ is an almost complex
submanifold. The set of such $J$ is nonempty and forms a
contractible space. With such a choice of almost complex
structure, $X_0$ can be viewed as a neighborhood of the zero
section of the complex line bundle $N_{Z|X}$ with the $S^1$ action
given by  the complex multiplication $e^{2\pi i \theta}$. Now we
remove $Z$. The end of $X-Z$ is simply $X_0-Z$. Recall that the
punctured disc $D-\{0\}$ is biholomorphic to the half cylinder
$S^1\times [0, \infty)$. Therefore, as an almost complex manifold,
$X_0-Z$ can be viewed as the translation invariant almost complex
half cylinder $P\times  [0, \infty)$ where $P=H^{-1}(0)$. In this
sense, $X-Z$ is viewed as a manifold with almost complex cylinder
end.

Now we consider a holomorphic map in the cylinder model where the
marked points mapped into $Z$ are removed from the domain surface.
Again we can view a punctured neighborhood of each of these marked
points as a half cylinder $S^1 \times [0,\infty)$. With such a
$J$, a $J-$holomorphic map of $X$ intersecting $Z$ at finitely
many points then exactly corresponds to a $J-$holomorphic map to
the open manifold $X-Z$ from a punctured Riemann surface which
converges to (a multiple of) an $S^1-$orbit at a puncture point.

Now we reconsider the convergence of $(\Sigma_n,f_n)$ in the
cylinder model. The creation of a $Z-$component $f_\nu$
corresponds to disappearance of a part of im$(f_n)$ to infinity.
We can use translation to rescale back the missing part of
im$(f_n)$. In the limit, we may obtain a stable map
$\tilde{f}_\nu$ into $P\times {\mathbb R}$. When we obtain $X$
from the cylinder model, we need to collapse the $S^1$-action at
infinity. Therefore, in the limit, we need to take into account
maps into the closure of $P\times {\mathbb R}$. Let $Y$ be the
projective completion of the normal bundle $N_{Z|X}$, i.e. $Y =
{\mathbb P}(N_{Z|X}\oplus {\mathbb C})$. Then $Y$ has a zero
section $Z_0$ and an infinity section $Z_\infty$. We view
$Z\subset X$ as the zero section.   One can further show that
$\tilde{f}_\nu$ is indeed  a stable map into $Y$ with the
stability specified below.

To form a compact moduli space of such maps we thus must allow the
target $X$ to degenerate as well (compare with \cite{Li1}). For
any non-negative integer $m$, construct $Y_m$ by gluing together
$m$ copies of $Y$, where the infinity section of the $i^{th}$
component is glued to the zero section of the $(i+1)^{st}$
component for $1\leq i \leq m$. Denote the zero section of the
$i^{th}$ component by $Z_{i-1}$, and the infinity section by
$Z_i$, so Sing $Y_m = \cup_{i=1}^{m-1} Z_i$. We will also denote
$Z_m$ by $Z_\infty$ if there is no possible confusion. Define
$X_m$ by gluing $ X $ along $Z$ to $Y_m$ along $Z_0$. Thus Sing $
X_m = \cup_{i=0}^{m-1}Z_i$ and $X_0 = X$. $X_0=X$ will be referred
to as the root component and the other irreducible components will
be called the bubble components. Let $\mbox{Aut}_Z Y_m$ be the
group of automorphisms of $Y_m$ preserving $Z_0$, $Z_m$, and the
morphism to $Z$.  And let $\mbox{Aut}_ZX_m$ be the group of
automorphisms of $X_m$ preserving $X$ (and $Z$) and with
restriction to $Y_m$ being contained in $\mbox{Aut}_ZY_m$ (so
$\mbox{Aut}_ZX_m = \mbox{Aut}_ZY_m\cong ({\mathbb C}^*)^m$, where
each factor of $({\mathbb C}^*)^m$ dilates the fibers of the
${\mathbb P}^1-$bundle $Y_i\longrightarrow Z_i$). Denote by
$\pi[m] : X_m\longrightarrow X$ the map which is the identity on
the root component $X_0$ and contracts all the bubble components
to $Z_0$ via the  fiber bundle projections.

Now consider a nodal curve $C$ mapped into $X_m$ by
$f:C\longrightarrow X_m$ with specified tangency to $Z$.  There
are two types of marked points:

(i) absolute marked points whose images under $f$ lie outside $Z$,
labeled by $x_i$,

(ii) relative marked points which are mapped into $Z$ by $f$,
labeled by $y_j$.

A relative $J-$holomorphic map $f:C\longrightarrow X_m$ is said to
be pre-deformable if $f^{-1}(Z_i)$ consists of a union of nodes
such that for each node $p\in f^{-1}(Z_i), i=1,2, \cdots, m$, the
two branches at the node are mapped to different irreducible
components of $X_m$ and the orders of contact to $Z_i$ are equal.

An isomorphism of two such $J-$holomorphic maps $f$ and $ f'$ to
$X_m$ consists of a diagram
$$
\begin{array}{ccc}
     (C,x_1, \cdots, x_l, y_1, \cdots, y_k) &
     \stackrel{f}{\longrightarrow} &  X_m\\
     & &\\
     h\downarrow &  & \downarrow t\\
     & &\\
     (C',x'_1, \cdots, x'_l, y'_1, \cdots, y'_k) &
     \stackrel{f'}{\longrightarrow} &  X_m
\end{array}
$$
where $h$ is an isomorphism of marked curves and $t\in
\mbox{Aut}_Z(X_m)$. With the preceding understood, a relative
$J-$holomorphic map to $X_m$ is said to be stable if it has only
finitely many automorphisms.

We introduced the notion of a weighted graph in Remark
\ref{absolute graph}. We need to refine it for relative stable
maps to $(X,Z)$. A (connected) relative graph $\Gamma$ consists of
the following data:

(1) a vertex decorated by $A\in H_2(X;\mathbb Z)$ and genus $g$,

(2) a tail for each absolute marked point,

(3) a relative tail for each relative marked point.

\begin{definition} Let $\Gamma$ be a relative graph
with $k$ (ordered) relative tails  and $T_k = (t_1, \cdots, t_k)$,
a $k-$tuple of positive integers forming a
 partition of $A\cdot Z$.
 A relative
$J-$holomorphic map to $(X, Z)$ with  type $(\Gamma, T_k)$
consists of a marked curve $(C, x_1, \cdots, x_l, y_1, \cdots,
y_k)$ and a map $f: C\longrightarrow X_m$ for some non-negative
integer $m$ such that

(i) $C$ is a connected curve (possibly reducible) of arithmetic
genus $g$,

(ii) the map
$$
   \pi_m\circ f : C \longrightarrow X_m \longrightarrow X
$$
satisfies $(\pi_m\circ f)_*[C] = A$,

(iii) the  $x_i, 1\leq i \leq l$, are the absolute marked points,

(iv) the $y_i, 1\leq i \leq k$, are the relative marked points,

(v) $f^*Z_m = \sum_{i=1}^k t_i y_i$.

\end{definition}

Let $\overline{\mathcal M}_{\Gamma, T_k}(X,Z,J)$ be the moduli
space of pre-deformable relative stable $J-$holomorphic maps with
type $(\Gamma, T_k)$. Notice that for an element $f:C\to X_m$ in
$\overline{\mathcal M}_{\Gamma, T_k}(X,Z,J)$ the intersection
pattern with $Z_0, ..., Z_{m-1}$ is only constrained by the genus
condition and the pre-deformability condition.

Now we apply the virtual neighborhood technique to construct
$U^{X,Z}_{{\mathcal S}_e}$, $\mathcal E_{X,Z}$, $S_{X,Z}$ as in
section \ref{gw}. Consider the configuration space
$\overline{\mathcal B}_\Gamma(X,Z,J)$ of equivalence classes of
smooth pre-deformable relative stable maps. Here we still take the
equivalence class under ${\mathbb C}^*$-action on the fibers of
${\mathbb P}(N_{Z|X}\oplus {\mathbb C})$. In particular, the
subgroup of ${\mathbb C}^*$ fixing such a map is required to be
finite. The maps are required to intersect the $Z_i$ only at
finitely many points in the domain curve. Further, at these
points, the map is required to have a holomorphic leading term  in
the normal Taylor expansion for any local chart of $X$ taking $Z$
to a coordinate hyperplane and being holomorphic in the normal
direction along $Z$. Thus the notion of contact order still makes
sense, and we can still impose the
 pre-deformability condition and contact order condition at the $y_i$
 being  governed by $T_k$.

Next, we can define $\Omega^1$ similarly. We also need to
understand the  $Obs$ space. The discussion is similar to that of
stable maps.

According to its label, a relative stable map is naturally divided
into components of two types:

(i) a stable map in $X$ intersecting $Z$ transversely, called a
rigid factor;

(ii) a stable map in ${\mathbb P}(N_{Z|X}\oplus {\mathbb C})$ such
that its projection to $Z$ is stable and it intersects $Z_0$,
$Z_\infty$ transversely, called a  rubber factor.

For each component $(\Sigma,f)$, we also have the extra condition
that $f$ intersects $Z_0$ or $Z_\infty$ with an order fixed by the
graph.

Suppose that $f(y_i)\in Z_0$ or $Z_\infty$ with order $t_i$. The
analog of (\ref{virtual-0}) is
$$
  L^{X,Z}_{\Sigma, f} : \Omega^0_r \longrightarrow \Omega^{0,1}_r.
$$
Here an element $u\in \Omega^0_r$ is an element of
$\Omega^0(f^*TX)$ or $\Omega^0(f^*T{\mathbb P}(N_{Z|X}\oplus
{\mathbb C}))$ with {\bf the following property:} Choose a unitary
connection on $N$ so that we can decompose the tangent bundle of
${\mathbb P}(N_{Z|X}\oplus {\mathbb C})$  into tangent and normal
directions. Near $f(y_i)$, $u(y_i)$ can be decomposed into $(u_Z,
u_N)$ where $u_Z$, $u_N$ are tangent and normal components,
respectively. Now we require that $u_N$ vanish at $y_i$ with order
$t_i$. When $\Sigma$  consists of  two components joined at one
point, we require their $u_Z$ components  be the same at the
intersection point.

We can also consider the analog of $\tilde{L}_{\Sigma,f}$,
$$
  \tilde{L}^{X,Z}_{\Sigma,f}\oplus \oplus _i T_{t_i} :
  \tilde{\Omega}^0_r \longrightarrow \tilde{\Omega}^{0,1}_r \oplus
  \oplus_i{\mathcal J}^{t_i}_i.
$$
Here $ \tilde{\Omega}^0_r = \{ u \in \Omega^0\mid u_N(y_i) = 0\}.
$ Also an element $v\in\tilde \Omega^{0,1}_r$ is required to have
a simple pole at each $y_i$ such that $\mbox{Res}_{y_i}v \in TZ$,
and  if  two components are joined together, we require that the
sum of the residues be zero. Each summand
$$
{\mathcal J}^{t_i}_i \cong \oplus^{t_i
-1}_{j=1}\mbox{Hom}((T_{y_i}\Sigma)^j, N_{f(y_i)})
$$
is the $(t_i-1)$-jet space, and the map $T_{t_i}(f)$ is the
$(t_i-1)$-jet of $f$ at $y_i$, i.e. the first $(t_i-1)$ terms of
the Taylor polynomial.

It is clear that  $\mbox{Coker}\, \tilde{L}^{X,Z}_\Sigma$ has a
similar description as $\mbox{Coker}\, \tilde{L}^X_\Sigma$, the
only difference being that we require the residue at each nodal
point be in $TZ$. Moreover,
$$
 \tilde{L}^{X,Z}_{\Sigma,f} \oplus \oplus_i T_{t_i} =
 \tilde{L}^{X,Z}_{\Sigma,f}\oplus \oplus_i
 T_{t_i}\mid_{\mbox{Ker}\,\tilde{L}^{X,Z}_{\Sigma,f}}.
$$

Finally the process of  adding the deformation of a nodal Riemann
surface is identical.

In addition to the  evaluation maps on $\overline{\mathcal
B}_{\Gamma, T_k}(X,Z,J)$,
$$\begin{array}{lllll}
 ev^X_i: &\overline{\mathcal B}_{\Gamma, T_k}(X,Z,J)& \longrightarrow &X, &\quad 1\leq i\leq l,\\
&(\Sigma, x_1,\cdots, x_l,y_1,\cdots, y_k,f)&\longrightarrow
&f(x_i),&
\end{array}
$$
there are also the evaluations maps
$$\begin{array}{lllll}
  ev^Z_j: &\overline{\mathcal B}_{\Gamma, T_k}(X,Z,J) &\longrightarrow& Z, &\quad  1\leq j\leq
  k,\\
  &(\Sigma, x_1,\cdots, x_l,y_1,\cdots, y_k,f)&\longrightarrow & f(y_j),&
\end{array}
$$
where $Z=Z_m$ if the target of $f$ is $X_m$.

\begin{definition} Let $\alpha_i\in H^*(X;{\mathbb R})$,$1\leq i\leq l$, $\beta_j\in
H^*(Z;{\mathbb R})$, $1\leq j\leq k$. Define the  relative
Gromov-Witten invariant
$$
   \langle\Pi_i\tau_{d_i}\alpha_i\mid \Pi_j\beta_j\rangle^{X,Z}_{\Gamma, T_k} =
   \frac{1}{|\hbox{Aut}(T_k)|}\int_{U^{X,Z}_{{\mathcal
   S}_e}}S^*_{X,Z}\Theta\wedge\Pi_i\psi_i^{d_i}\wedge(ev^X_i)^*\alpha_i\wedge\Pi_j(ev^Z_j)^*\beta_j,
$$
where $\Theta$ is the Thom class of the bundle ${\mathcal
E}_{X,Z}$ and Aut$(T_k)$ is the symmetry group of the partition
$T_k$. Denote by ${\mathcal T}_k = \{ (t_j, \beta_j)\mid j=1,
\cdots, k\}$ the weighted partition of $A\cdot Z$. If the vertex
of $\Gamma$ is decorated by $(g, A)$,  we will sometimes write
$$
   \langle\Pi_i\tau_{d_i}\alpha_i\mid {\mathcal
  T}_k\rangle^{X,Z}_{g,A}
  $$
  for
  $\langle\Pi_i\tau_{d_i}\alpha_i\mid
\Pi_j\beta_j\rangle^{X,Z}_{\Gamma, T_k} $.
\end{definition}

\begin{remark} In \cite{LR} only invariants  without
descendent classes were considered. But it is straightforward  to
extend the definition of \cite{LR} to include descendent classes.

\end{remark}

We can  decorate the tail of a relative graph $\Gamma$ by $(d_i,
\alpha_i)$ as in the absolute case. We can further  decorate the
relative tails by the weighted partition ${\mathcal T}_k$. Denote
the resulting weighted relative graph by
$\Gamma\{(d_i,\alpha_i)\}| {\mathcal T}_k$. In \cite{LR} the
source curve is required to be connected. We will also need to use
a disconnected version. For a disjoint union $\Gamma^{\bullet}$ of
weighted relative  graphs and a corresponding disjoint union of
partitions, still denoted by $T_k$, we  use
$\langle\Gamma^{\bullet}\{(d_i, \alpha_i)\} | {\mathcal T}_k
\rangle^{X,Z}$ to denote the corresponding relative invariants
with a disconnected domain, which is simply the product of the
connected relative invariants. Notice that although we use
$\bullet$ in our notation following \cite{MP}, our disconnected
invariants are different. The disconnected invariants there depend
only on the genus, while ours depend on the finer graph data.

\subsection{Partial orderings on relative GW
invariants}

In \cite{MP}, the authors first introduced a partial order on the
set of relative Gromov-Witten invariants of a ${\mathbb
P}^1$-bundle. The authors, \cite{HLR}, refined their partial order
on the set of relative Gromov-Witten invariants of a Blow-up
manifold relative to the exceptional divisor, and used this
partial order to obtain a Blow-up correspondence of
absolute/relative Gromov-Witten theory. In this subsection, we
will review the partial order on the set of relative Gromov-Witten
invariants.

First of all, all Gromov-Witten invariants vanish if $A\in
H_2(X,{\mathbb Z})$ is not an effective curve class. We define a
partial ordering on $H_2(X,{\mathbb Z})$ as follows:
$$
   A'< A
$$
if $A-A'$ is a nonzero effective curve class.

The set of pairs $(m,\delta)$ where $m\in {\mathbb Z}_{>0}$ and
$\delta\in H^*(Z, {\mathbb Q})$ is partially ordered by the
following size relation
\begin{equation}\label{size-order}
   (m,\delta) > (m', \delta')
\end{equation}
if $m>m'$ or if $m=m'$ and $\deg(\delta)>\deg(\delta')$.

Let $\mu$ be a partition weighted by the cohomology of $Z$, i.e.,
$$
   \mu = \{(\mu_1, \delta_{r_1}),\cdots, (\mu_{\ell(\mu)},
   \delta_{r_{\ell(\mu)}})\}.
$$
We may place the pairs of $\mu$ in decreasing order by size
(\ref{size-order}). We define
$$
  \deg (\mu) = \sum \deg (\delta_{r_i}).
$$
A lexicographic ordering on weighted partitions is defined as
follows:
$$
   \mu \stackrel{l}{>}\mu'
$$
if, after placing the pairs in $\mu$ and $\mu'$ in decreasing
order by size, the first pair for which $\mu$ and $\mu'$ differ in
size is larger for $\mu$.

For the nondescendent relative Gromov-Witten invariant
$$
   \langle \varpi\mid \mu \rangle^{X,Z}_{g,A},
$$
denote by $\| \varpi\|$ the number of absolute insertions.

\begin{definition}\label{order}
A partial ordering $\stackrel{\circ}{<}$ on the set of
nondescendent relative Gromov-Witten invariants is defined as
follows:
$$
  \langle \varpi'\mid \mu'\rangle^{X,Z}_{g',A'}
  \stackrel{\circ}{<} \langle \varpi\mid \mu\rangle^{X,Z}_{g,A}
$$
if one of the conditions below holds
\begin{enumerate}
\item[(a)] $A'<A$,
\item[(b)] equality in (a) and $g'<g$,
\item[(c)] equality in (a)-(b) and $\|\varpi'\|<\|\varpi\|$,
\item[(d)] equality in (a)-(c) and $\deg(\mu') > \deg (\mu)$,
\item[(e)] equality in (a)-(d) and $\mu'\stackrel{l}{>}\mu$.
\end{enumerate}

\end{definition}

\subsection{Degeneration formula}\label{df}
Now we describe the degeneration formula of GW-invariants under
symplectic cutting.

As an operation on topological spaces,  the symplectic cut is
essentially collapsing the circle orbits in the hypersurface
$H^{-1}(0)$ to points in $Z$.

Suppose that $X_0\subset X$ is an open codimension zero
submanifold with a Hamiltonian $S^1-$action. Let $H:X_0\to \mathbb
R$ be a Hamiltonian function with $0$ as a regular value. If
$H^{-1}(0)$ is a separating hypersurface of $X_0$, then we obtain
two connected manifolds $X_0^{\pm}$ with boundary $\partial
X_0^{\pm}=H^{-1}(0)$. Suppose further that $S^1$ acts freely on
$H^{-1}(0)$. Then the  symplectic reduction $Z=H^{-1}(0)/S^1$ is
canonically a symplectic manifold of dimension $2$ less.
Collapsing the $S^1-$action on $\partial X^{\pm}=H^{-1}(0)$, we
obtain closed smooth manifolds $\overline{X_0}^{\pm} $ containing
respectively real codimension $2$ submanifolds $Z^{\pm}=Z$ with
opposite normal bundles. Furthermore $\overline {X_0}^{\pm}$
admits a symplectic structure $\overline \omega^{\pm}$ which
agrees with the restriction of $\omega$ away from $Z$, and whose
restriction to $Z^{\pm}$ agrees with the canonical symplectic
structure $\omega_Z$ on $Z$ from symplectic reduction.

This is neatly shown by considering $X_0\times \mathbb C$ equipped
with appropriate  product symplectic structures and
 the product $S^1$-action on $X_0\times\mathbb C$, where $S^1$ acts on $\mathbb C$ by complex
multiplication. The extended action is Hamiltonian if we use the
standard symplectic structure $\sqrt {-1}dw\wedge d\bar w$ or its
negative on the $\mathbb C$ factor.Then the moment map is
$$\mu_+(u, w)=H(u)+|w|^2:X_0\times \mathbb C \to {\mathbb R},$$
 and
$\mu_+^{-1}(0)$ is the disjoint union of $S^1-$invariant sets
$$\{(u,w)| H(u)= -|w|^2<0\} \quad\hbox{and}\quad \{(u, 0)|H(z)=0\}.$$ We
define $\overline {X_0}^+$ to be the symplectic reduction
$\mu_+^{-1}(0)/S^1$. Then $\overline X_0^+$ is the disjoint union
of an open symplectic submanifold and a closed codimension 2
symplectic submanifold identified with $(Z, \omega_Z)$. The open
piece can be identified symplectically with
$$X_0^+=\{u\in X_0|H(u)<0\}\subset X_0$$
by the map $u\to (u, \sqrt {-H(u)}).$

Similarly, if we use $-{\bf i}dw\wedge d\bar w$, then the moment
map is
$$\mu_-(u, w)=H(u)-|w|^2:X_0\times \mathbb C \to \mathbb R$$
and  the corresponding symplectic reduction $\mu_-^{-1}(0)/S^1$,
denoted by $\overline X_0^-$,  is the disjoint union of an open
piece identified symplectically with
$$
X_0^-=\{u\in X_0|H(u)>0\}
$$
by the map $\phi_0^-:u\to (u, \sqrt {H(u)})$, and a closed
codimension 2 symplectic submanifold identified with $(Z,
\omega_Z)$.

We finally define $\overline X^{+}$ and $\overline X^{-}$.
$\overline X^+$ is simply $\overline{X_0}^+$, while $\overline
X^-$ is obtained from gluing symplectically $X^-$ and $\overline
X_0^-$ along $X_0$ via $\phi_0^-$. Notice that $\overline
X^-=(X^--X_0)\cup \overline {X_0}^-$ as a set.

 The two symplectic
manifolds $(\overline{X}^{\pm}, \overline \omega^{\pm})$ are
called the symplectic cuts of $X$ along $H^{-1}(0)$.

Thus we have a continuous map
$$\pi:X\to \overline{X}^+\cup_Z\overline{X}^-.$$
As for the symplectic forms, we have $\omega^+|_Z = \omega^-|_Z$.
Hence, the pair $(\omega^+, \omega^-)$ defines a cohomology class
of $\overline{X}^+\cup_Z\overline{X}^-$, denoted by
$[\omega^+\cup_Z\omega^-]$. It is easy to observe that
\begin{equation}\label{cohomology relation}
   \pi^* ([\omega^+\cup_Z\omega^-]) = [\omega].
\end{equation}
 Let $B\in
H_2(X;{\mathbb Z})$ be in the kernel of
$$
  \pi_* : H_2(X;{\mathbb Z})
\longrightarrow H_2(\overline{X}^+\cup_Z\overline{X}^-; {\mathbb
Z}). $$
 By (\ref{cohomology relation}) we have $\omega(B) =0$.
Such a class is called a vanishing cycle.
 For $A\in H_2(X; {\mathbb Z})$ define $[A] = A + \mbox{Ker}
(\pi_*)$ and
\begin{equation}\label{vanishing cycle}
 \langle\Pi_i\tau_{d_i}\alpha_i\rangle^X_{g,[A]}
= \sum_{B\in[A]}\langle\Pi_i\tau_{d_i}\alpha_i\rangle^X_{g,B} .
\end{equation}
Notice that  $\omega$ has constant pairing with any element in
$[A]$.
 It follows from  the Gromov
compactness theorem that there are only finitely many such
elements in $[A]$ represented by $J$-holomorphic stable maps.
Therefore, the summation in (\ref{vanishing cycle}) is finite.

The degeneration formula expresses
$\langle\Pi_i\tau_{d_i}\alpha_i\rangle^X_{g,[A]} $ in terms of
relative invariants of $(\overline{X}^+, Z)$ and $(\overline{X}^-,
Z)$ possibly with disconnected domains.

To begin with, we need to assume that the cohomology class
$\alpha_i$ is of the form $$\alpha_i =
\pi^*(\alpha^+\cup_Z\alpha^-).$$
 Here $\alpha_i^\pm \in H^*(\overline{X}^\pm; {\mathbb
R})$ are classes with  $\alpha_i^+|_Z = \alpha_i^-|_Z$ so that
 they give rise to a class $\alpha_i^+\cup_Z\alpha_i^-\in
H^*(\overline{X}^+\cup_Z\overline{X}^-; {\mathbb R})$.

Next, we proceed to write down the degeneration formula.
    We first specify the relevant
topological type of a marked Riemann surface mapped into
$\overline{X}^+\cup_Z\overline{X}^-$ with the following
properties:
\begin{enumerate}
\item[(i)] Each connected component is mapped  either into
$\overline{X}^+$ or $\overline{X}^-$ and carries a respective
degree 2 homology class; \item[(ii)] The images of two distinct
connected components only intersect each other along $Z$;
\item[(iii)] No two connected components which are both mapped
into $\overline{X}^+$ or $\overline{X}^-$ intersect each other;
\item[(iv)] The marked points are not mapped to $Z$; \item[(v)]
Each point in the domain mapped to $Z$ carries a positive integer
(representing the order of tangency).
\end{enumerate}

By abuse of language we call the above data a $(\overline{X}^+,
\overline{X}^-)-$graph.
 Such a graph gives rise to  two relative
graphs of $(\overline{X}^+, Z)$ and $\overline{X}^-, Z)$ from
(i-iv), each possibly being disconnected. We denote them by
 $\Gamma^{\bullet}_+$ and $
\Gamma^{\bullet}_-$ respectively. From (v) we also get  two
partitions $T_+ $ and $T_-$.  We call   a $(\overline{X}^+,
\overline{X}^-)-$graph a degenerate $(g, A, l)-$graph if the
resulting pairs $(\Gamma^{\bullet}_+, T_+)$ and $
(\Gamma^{\bullet}_-, T_-)$ satisfy the following constraints: the
total number of marked points is $l$, the relative tails are the
same, i.e. $ T_+= T_-$, and the identification of relative tails
produces a connected graph of $X$ with  total homology class
$\pi_*[A]$ and arithmetic genus $g$.

Let $\{\beta_a\}$ be a self-dual basis of $H^*(Z;{\mathbb R})$ and
$\eta^{ab} = \int_Z\beta_a\cup \beta_b$. Given $g, A$ and $ l$,
consider a degenerate  $(g, A, l)-$graph.
 Let $T_k=T_+=T_-$ and ${\mathcal T}_k$ be
a weighted partition $\{t_j, \beta_{a_j}\}$. Let ${\mathcal T}'_k
=\{t_j, \beta_{a_{j'}}\}$ be the dual weighted partition.

The degeneration formula for
$\langle\Pi_i\tau_{d_i}\alpha_i\rangle^X_{g,[A]}$ then reads as
follows,
$$
\begin{array}{ll}
 & \langle\Pi_i\tau_{d_i}\alpha_i\rangle^X_{g,[A]}
\cr =&\sum
   \langle \Gamma^{\bullet} \{(d_i, \alpha_i^+)\}|   {\mathcal T}_k\rangle^{\overline{X}^+,Z}
   \Delta({\mathcal T}_k) \langle \Gamma^{\bullet}\{(d_i, \alpha_i^-)\}|
    {\mathcal T}'_k\rangle^{\overline{X}^-,Z},\cr
\end{array}
$$
where the summation is taken over all degenerate $(g,A,
l)-$graphs, and
$$\Delta({\mathcal T}_k )=\prod_j t_j |\mbox{Aut}\,(T_k)|.$$

\section{Relative GW-invariants of ${\mathbb P}^1$-bundles}

Suppose that $Z$ is a symplectic submanifold of $X$ of codimension
2. When applying the degeneration formula, we often need to
express the absolute Gromov-Witten invariants of $X$ as a
summation of products of relative Gromov-Witten invariants of
symplectic cuts of $X$. Thus if we want to obtain a comparison
theorem of Gromov-Witten invariant by the degeneration formula,
the point will be how to compute the relative Gromov-Witten
invariants of a ${\mathbb P}^1$-bundle. In this section, we will
prove a vanishing theorem for genus zero relative Gromov-Witten
invariants of the ${\mathbb P}^1$-bundle $Y$ relative to the
infinity section and compute some genus zero two-point relative
fiber class GW invariants of the ${\mathbb P}^1$-bundle $Y$.

Suppose that $L$ is a line bundle over $Z$ and $Y={\mathbb
P}(L\oplus {\mathbb C})$. Let $D$, $Z$ be the infinity section and
zero section of $Y$ respectively. Let $\beta_1, \cdots,
\beta_{m_Z}$ be a self-dual basis of $H^*(Z,{\mathbb Q})$
containing the identity element. We will often denote the identity
by $\beta_{\rm id}$. The degree of $\beta_i$ is the real grading
in $H^*(Z,{\mathbb Q})$. We view $\beta_i$ as an element of
$H^*(Y, {\mathbb Q})$ via the pull-back by $\pi: Y={\mathbb
P}(L\oplus {\mathbb C})\longrightarrow Z$. Let $[Z]$, $[D]\in
H^2(Y, {\mathbb Q})$ denote the cohomology classes associated to
the divisors. Define classes in $H^*(Y, {\mathbb Q})$ by
\begin{eqnarray*}
\gamma_i &=& \beta_i,\\
\gamma_{m_Z+i} &=& \beta_i\cdot [Z],\\
\gamma_{2m_Z+i} &=& \beta_i\cdot [D].
\end{eqnarray*}
We will use the following notation:
\begin{eqnarray*}
\gamma_i^\beta &=& \beta_{i\,\,\mbox{mod}\,\, m_Z}, \\
\gamma_i^D &=& 1, [Z],\mbox{or} [D].
\end{eqnarray*}
The second assignment depends upon the integer part of
$(i-1)/m_Z$. The set $\{\gamma_1,\cdots,\gamma_{2m_Z}\}$
determines a basis of $H^*(Y,{\mathbb Q})$.

Since $D\cong Z$ as topological manifolds, $H^*(D, {\mathbb Q})$ is
isomorphic to $H^*(Z,{\mathbb Q})$. If there is no confusion, we
will also use $\beta_1,\cdots, \beta_{m_Z}$ as a self-dual basis of
$H^*(D,{\mathbb Q})$.

\subsection{Fiber class invariants}
In this subsection, we mainly compute some genus zero relative GW
invariants of ${\mathbb P}^1$-bundles with a fiber class.
According to \cite{MP}, we may transfer the computation of this
invariant on the  ${\mathbb P}^1$-bundle into that of some
associated invariants on ${\mathbb P}^1$.

Suppose that $L$ is a line bundle over $Z$ and
$$
 Y= {\mathbb P}(L\oplus {\mathcal O}) = \{ (z,l)| z\in Z, l\subset L_z\oplus
{\bf C}\},
$$
where $L_z$ is the fiber of $L$ at $z$.

Denote by $\pi: Y\longrightarrow Z$ the projection of the
${\mathbb P}^1$-bundle $Y$. Moreover there are two inclusions
(sections) of $Z$ in $Y= {\mathbb P}(L\oplus {\mathcal O})$:
\begin{enumerate}
\item[(1)] the ``zero section" $z\longrightarrow (z,0\oplus {\bf
C})$, denoted by $Z$,

\item[(2)] the ``section at infinity" $z\longrightarrow (z,
L_z\oplus 0)$, denoted by $D$.
\end{enumerate}

Let $\Gamma$ be the relative graph with the following data
\begin{enumerate}
\item[(1)] a vertex decorated by $A= sF\in H_2(Y, {\mathbb Z})$
and genus zero;

\item[(2)] $k$ relative tails;

\item[(3)] $l$ absolute tails.
\end{enumerate}

For any non-negative integer $m$, define $Y_m$ by gluing together
$m$ copies of $Y$, where the infinity section of the $i^{th}$
component is glued to the zero section of the $(i+1)^{st}$ $(1\leq
i \leq m)$ component; see Section \ref{rgw} for details.  Denote
by $\pi[m] : Y_m\longrightarrow Y$ the map which is the identity
on the root component $Y_0$ and contract all the bubble components
to $D_0$ via the projection of the fiber bundle of $Y_i$.

 Let $T_k=\{t_1,\cdots,t_k\}$ be a $k$-tuple of positive integers forming a partition of $s$.
 Denote by $\overline{\mathcal M}_{\Gamma, T_k}(Y,D) $  the
 moduli space of morphisms
 $$
       f:(C,x_0,\cdots, x_l; y_1,\cdots,y_k) \longrightarrow ( Y_m, D_\infty),
 $$
 such that
\begin{enumerate}
 \item[(1)] $(C,x_0,\cdots,x_l; y_1,\cdots,y_k)$ is a prestable curve of genus zero with
 $l$ absolute marked points $x_0,\cdots,x_l$ and $k$ relative marked points $y_1,\cdots, y_k$;.

 \item[(2)] $f^{-1}(D_\infty) = \sum t_iy_i$ as Cartier divisor and $\deg
 (\pi[m]\circ f) = s$.

\item[(3)] The predeformability condition: The preimage of the
singular
 locus Sing $Y_m = \cup_{i=0}^{m-1}D_i$ of $Y_m$
 is a union of  nodes of $C$, and if $p$ is one such node, then the two branches
 of $C$ at $p$ map into different irreducible components of $ Y_m$,
 and their orders of contact with the divisor $D_i$ (in their
 respective components of $Y_m$) are equal. The morphism $f$ is
 also required to satisfy a stability condition that there are no
 infinitesimal automorphisms of the sequence of maps $(C, x_0,\cdots,x_l; y_1,\cdots,y_k)\longrightarrow Y_m
 \stackrel{\pi[m]}{\longrightarrow}Y$ where the allowed
 automorphisms of the map from $Y_m$ to $Y$ are
 $\mbox{Aut}_D(Y_m)$.

 \item[(4)] The automorphism group of $f$ is finite.
\end{enumerate}

 Two such morphisms are isomorphic if they differ by an
 isomorphism of the domain and an automorphism of $(Y_m,D_0,
 D_\infty)$. In particular, this defines the automorphism group in
 the stability condition $(4)$ above.

 We introduce some notations which are used in \cite{Li1}. For any
 non-negative integer $m$, let
 $$
   {\mathbb P}^1[m] = {\mathbb P}^1_{(0)}\cup {\mathbb P}^1_{(1)}\cup \cdots
   \cup {\mathbb P}^1_{(m)}
 $$
 be a chain of $m+1$ copies ${\mathbb P}^1$, where ${\mathbb P}^1_{(l)}$ is
 glued to ${\mathbb P}^1_{(l+1)}$ at $p_1^{(l)}$ for $0\leq l\leq m-1$.
 The irreducible component ${\mathbb P}^1_{(0)}$ will also be
 referred to
 as the root component and the other irreducible components will be
 called the bubble components. A point $p_1^{(m)}\not= p_1^{(m-1)}$
 is fixed on ${\mathbb P}^1_{(m)}$. Denote still by
 $\pi[m]:{\mathbb P}^1[m]\longrightarrow {\mathbb P}^1$ the map which is the identity on the
 root component and contracts all the bubble components to
 $p_1^{(0)}$. For $m>0$, let
 $$
   {\mathbb P}^1(m) = {\mathbb P}^1_{(1)}\cup\cdots\cup {\mathbb P}^1_{(m)}
 $$
 denote the union of bubble components of ${\mathbb P}^1[m]$.

 Similar to the case of $Y_m$, we may define the associated moduli
 space $\overline{\mathcal M}_{\Gamma}({\mathbb P}^1, p_1^{(0)}; T_1)$ of relative stable maps
 to $({\mathbb P}^1, p_1^{(0)})$, see \cite{Li1} for its definition.

Next, we first review Maulik-Pandharipande's algorithm \cite{MP}
which reduces the relative Gromov-Witten invariant of $(Y,D)$ of
fiber class to that of $({\mathbb P}^1,p_1)$. Note that the moduli
space of stable relative maps
$$
    \overline{\mathcal M}_Y = \overline{\mathcal
    M}_{\Gamma, T_k}(Y,D)
$$
is fibered over $Z$,
\begin{equation}\label{fiber}
   \pi : \overline{\mathcal M}_Y \longrightarrow Z,
\end{equation}
with fiber isomorphic to the moduli space of maps of degree $s$ to
${\mathbb P}^1$ relative to the infinity point $p_1$ with tangency
order $s$:
$$
    \overline{\mathcal M}_{{\mathbb P}^1} = \overline{\mathcal
    M}_{\Gamma,T_k}({\mathbb P}^1,p_1).
$$
In fact, $\overline{\mathcal M}_Y$ is the fiber bundle constructed
from the principal $S^1$-bundle associated to $L$ and a standard
$S^1$-action on $\overline{\mathcal M}_{{\mathbb P}^1}$.

The $\pi$-relative obstruction theory of $\overline{\mathcal M}_Y$
is obtained from the $\overline{\mathcal M}_{{\mathbb P}^1}$-fiber
bundle structure over $Z$. The relationship between the
$\pi$-relative virtual fundamental class $[\overline{\mathcal
M}_{\bar{M}^+}]^{vir_\pi}$ and the virtual fundamental class
$[\overline{\mathcal M}_{\bar{M}^+}]^{vir}$ is given by the
equation
\begin{equation}\label{virtual}
[\overline{\mathcal M}_Y]^{vir} = c_{top}({\mathbb E}\otimes
TZ)\cap [\overline{\mathcal M}_Y]^{vir_\pi}
\end{equation}
where ${\mathbb E}$ is the Hodge bundle. Since we only consider
the case of genus zero,  (\ref{virtual}) can be written as
$$
[\overline{\mathcal M}_Y]^{vir} = [\overline{\mathcal
M}_Y]^{vir_\pi}.
$$

By integrating along the fiber, we can compute the relative
Gromov-Witten invariants of $Y$ by computing the equivariant
integrations in the relative Gromov-Witten theory of ${\mathbb
P}^1$; see \cite{MP} for the details.

Let ${\mathcal T}_k = \{(t_i,\beta_i)\}$ be the cohomology
weighted partition of $s$. By definition, we have
\begin{eqnarray}\label{invariant}
   & & \langle \tau_{d_1-1}\gamma_1,\cdots, \tau_{d_l-1}\gamma_l\mid{\mathcal T}_k\rangle^{Y,D}_{\Gamma}
   = \int_{[\overline{\mathcal M}_Y]^{vir}}\prod_{i=1}^l \psi_i^{d_i-1} ev_i^*\gamma_i\wedge \prod_jev_j^*\beta_j\\
   &=&\frac{1}{|\mbox{Aut}(T_k)|}\int_Z(\prod_i\gamma_i^\delta\prod_j\beta_j\cap
   \pi_*(\prod_i\psi_i^{d_i-1}ev_i^*(\gamma_i^{D_0})\cap[M_Y]^{vir_\pi})),\nonumber
\end{eqnarray}
 where the interior push-forward
$$
\pi_*(\prod_i\psi_i^{d_i-1}ev_i^*(\gamma_i^{D_0})\cap[M_Y]^{vir_\pi})
$$
is obtained from the corresponding Hodge integral in the
equivariant Gromov-Witten theory of $({\mathbb P}^1,p_1)$ after
replacing the hyperplane class on ${\mathbb C}{\mathbb P}^\infty$
by $C_1(L)$.

Therefore, via (\ref{invariant}), we may reduce the computation of
relative Gromov-Witten invariants $\langle
\tau_{d_1-1}\gamma_1,\cdots, \tau_{d_l-1}\gamma_l\mid{\mathcal
T}_k\rangle^{Y,D}_{\Gamma}$ to that of
$$
  \langle
\tau_{d_1-1}\delta_1,\cdots, \tau_{d_l-1}\delta_l\mid pt,\cdots,
pt\rangle^{{\mathbb P}^1,p_1}_{\Gamma, T_k},
$$
where $\delta_i\in H^*({\mathbb P}^1, {\mathbb Q})$, $1\leq i \leq
l$.

About the two point genus zero relative Gromov-Witten invariant of
$({\mathbb P}^1,p_1)$, we have

\begin{lemma}\label{lemma3-4}Let $\varpi \in H^2({\mathbb P}^1,
{\mathbb Q})$.
\begin{enumerate}
\item[(i)]If $d\not= s$, then $ \langle \tau_{d-1}\varpi\mid
(s,[pt])\rangle^{{\mathbb
 P}^1, p_1}_s =0$.

 \item[(ii)]  For $s>0$, we have
$$
  \langle
\tau_{s-1}\varpi\mid (s,[pt])\rangle^{{\mathbb
 P}^1, p_1}_s= \frac{1}{s!}.
$$
\end{enumerate}
\end{lemma}

The proof of (i) follows from a simple dimension calculation and
(ii) of the lemma is Lemma 1.4 of \cite{OP}. In \cite{HLR}, the
authors generalized the result to general projective space
${\mathbb P}^n$ via localization techniques.

\begin{prop}\label{prop3-5}Let $s>0$.
\begin{enumerate}
\item[(i)] Let ${\mathcal T}_k = \{(t_i,\beta_i)\}$ be a
cohomology weighted  partition of $s$. Then
$$
    \langle \pi^*\alpha_1, \cdots, \pi^*\alpha_q,\beta_1\cdot [Z],\cdots,\beta_l\cdot [Z]\mid
    {\mathcal T}_k\rangle^{Y,D}_{sF} =0
$$
except for $s=k=1$ and $q=0$.

 \item[(ii)] For $s>0$, we have the two-point relative
invariant
$$
   \langle \tau_{d-1}(\beta_0\cdot [Z])\mid(s,\beta_\infty)\rangle^{Y,D}_{sF}
    =\left\{ \begin{array}{ll}
    \frac{1}{s!}\int_Z\beta_0\wedge\beta_\infty, & d=s\\
    &\\
    0, & d\not= s
    \end{array}\right.,
$$
where $\beta_0\in H^*(Z, {\mathbb Q})$ and $\beta_\infty\in H^*(D,
{\mathbb Q})$.

\item[(iii)] For $s=k=1$, we have
$$
  \langle \iota^!(\beta_1),\cdots, \iota^!(\beta_l) \mid (1,\gamma)\rangle^{Y,D}_{F}
  = \int_Z\beta_1\wedge\cdots\wedge\beta_l\wedge \gamma.
$$
\end{enumerate}
\end{prop}

\begin{proof}
(i). From (\ref{invariant}), we are reduced to a relative
Gromov-Witten invariant of ${\mathbb P}^1$ of the form
$$
  \langle {\mathbb P}^1, \cdots, {\mathbb P}^1, [pt], \cdots, [pt]\mid (t_1,[pt]),\cdots,
  (t_k,[pt])\rangle^{{\mathbb P}^1, p_1}_s.
$$
A dimension count shows that this invariant of ${\mathbb P}^1$ is
nonzero only if $s+k=2-q$. Since $s>0$ and $k>0$,  the only
possibility is $s=k=1$ and $q=0$.

The proof of (ii) directly follows from (\ref{invariant}) and
Lemma \ref{lemma3-4}.

 (iii). From (\ref{invariant}),  we have
\begin{eqnarray*}
& &\langle \iota^!(\beta_1),\cdots, \iota^!(\beta_l) \mid
(1,\gamma)\rangle^{Y,D}_{F}\\
 & = & \int_Z\beta_1\wedge\cdots\wedge\beta_l\wedge \gamma \langle [pt],\cdots,[pt] |(1,[pt])\rangle^{{\mathbb P}^1, p_1}_1.
\end{eqnarray*}
It remains to prove $\langle [pt],\cdots,[pt]
|(1,[pt])\rangle^{{\mathbb P}^1, p_1}_1=1 $. In fact, we consider
the absolute invariant of ${\mathbb P}^1$ with $l+1$ point
insertions: $\langle [pt],\cdots, [pt]\rangle^{{\mathbb P}^1}_1$.
First of all, by divisor axiom, we know that this absolute
invariant equals $1$. We apply the degeneration formula to this
invariant of  ${\mathbb P}^1$ and distribute one point insertion
to one side and other $l$ point insertions to other side. Then we
have
\begin{eqnarray*}
1 &=& \langle [pt],\cdots,[pt]\rangle^{{\mathbb P}^1}_1 \\
& =& \langle [pt],\cdots,[pt] |(1,[pt])\rangle^{{\mathbb P}^1,
p_1}_1\langle [pt]\mid (1, [pt])\rangle^{{\mathbb P}^1, p_1}_1\\
&=& \langle [pt],\cdots,[pt] |(1,[pt])\rangle^{{\mathbb P}^1,
p_1}_1.
\end{eqnarray*}
In the last equality, we used Lemma \ref{lemma3-4}. This proved
(iii).

\end{proof}

\subsection{A vanishing theorem}

In this subsection, we will prove a vanishing result for some
relative Gromov-Witten invariants of ${\mathbb P}^1$-bundle, in
particular, for some non-fiber homology class invariants.

Let $\Gamma_0$ be a relative graph with the following data:
\begin{enumerate}
\item[(i)] a vertex decorated by a homology class $A\in H_2(Y,
{\mathbb Q})$ and genus zero,

\item[(ii)] $l+q$ tails associated to $l+q$ absolute marked
points,

\item[(iii)]$k$ relative tails associated to $k$ relative marked
points.
\end{enumerate}

Denote by $A$ the homology class of the relative stable map
$(\Sigma, f)$ to $(Y,D)$ and by $F$ the homology class of a fiber
of $Y$. Let   $T_k= \{ t_1,\cdots,t_k\}$ be a partition of $D\cdot
A$ and $d_i$, $1\leq i\leq l$, be positive integers. Denote $d =
\sum^l_{i=1} d_i$. Denote by $\iota: Z\longrightarrow Y$ the
inclusion of $Z$ into $Y$ via the zero section of $Y$. Then for
any $\beta\in H^*(Z,{\mathbb R})$, the inclusion map $i$ pushes
forward the class $\beta$ to a cohomology class $\iota^!(\beta)\in
H^*(Y, {\mathbb Q})$, determined by the pull-back map $\iota^*$
and Poincar\'e duality.

\begin{prop}\label{prop3-1}
  Suppose $A\not= sF$ or $k+l+q\geq 3$. Assume
that $Z^*(A)\geq \sum d_i$ and $c_1(L)(C)\geq 0$ for any
$J$-holomorphic curve $C$ into $Z$. Then for any $\beta_i\in
H^*(Z,{\mathbb Q})$, $1\leq i\leq l$, and any weighted partition
${\mathcal T}_k=\{(t_i,\delta_i)\}$ of $D\cdot A$, we have
$$
   \langle \varpi, \tau_{d_1-1}\iota^!(\beta_1), \cdots,
   \tau_{d_l-1}\iota^!(\beta_l)\mid {\mathcal T}_k\rangle ^{Y,
   D}_{\Gamma_0, T_k} = 0,
$$
where $\varpi$ consists of insertions of the form $\pi^*\alpha_1,
\cdots, \pi^*\alpha_q$.
\end{prop}

\begin{proof}
  The projection $\pi: Y={\mathbb P}(L\oplus {\mathbb C})\longrightarrow Z$
induces a map between the   moduli spaces, denoted also by $\pi$,
\begin{equation}\label{proof-0}
  \pi: \overline{\mathcal B}_{\Gamma_0, T_k}(Y,D,J) \longrightarrow
  \overline{\mathcal B}^Z_{\pi_*(A)}(0,k+l+q,J),
\end{equation}
where $\pi$ contracts the unstable rational component whose image
is a fiber. $\pi$ is well-defined if $A\not= sF$ or $k+l+q\geq 3$.
By the definition, $\pi$ commutes with the evaluation map.
Furthermore, there is also a natural map(denoted by $\pi$ as well)
on $\Omega^{0,1}$ commuting with the map on configuration spaces.
Moreover, $\overline{\partial}$ commutes with $\pi$. Hence, it
induces a map from $\mbox{Coker}\,L^{Y,D}$ to $\mbox{Coker}\,L^Z$.
We claim that $\pi$ induces a map on a virtual neighborhood.

Let $\omega$ be an integral symplectic form on $Z$. Using
Siebert's construction \cite{S1}, we can construct a bundle
$\mathcal E$ dominating the local obstruction bundle generated by
$\mbox{Coker}\,L^Z_{\pi(\Sigma,f)}$. $\pi^*{\mathcal E}$ is a
bundle over $\overline{\mathcal B}_{\Gamma_0, T_k}(Y,D,J) $. We
want to show that $\pi^*\mathcal E$ dominates its local
obstruction bundle. Let $(\Sigma,f)$ be a relative stable map of
$(Y,D)$. Then we have

\begin{lemma}\label{lemma3-2}
  $\mbox{Coker}\, L^{Y,D}_{\Sigma,f}$ is isomorphic to
$\mbox{Coker}\, L^Z_{\pi(\Sigma,f)}$.
\end{lemma}

\begin{proof}
It is well-known that a stable map can be naturally decomposed
into connected components lying outside of $D$( rigid factors) or
completely inside $D$(rubber factors).
  Let $(\Sigma, f)$ be a rigid factor or
a rubber factor with relative marked points $x_1,\cdots, x_r$ such
that $f(x_i)\in Z$ or $D$ with order $k_i$. In both cases, it is a
stable map into $Y$. We take the complex as
$$
 \tilde{L}^{Y,D}_{\Sigma,f} \times \sum T^{k_i}_{x_i} :
 \{u\in\Omega^0(f^*TY)\mid u(x_i)\in TZ\} \longrightarrow
 \Omega^{0,1}(f^*TY\otimes_i{\mathcal O}_\Sigma(x_i))\oplus_i{\mathcal
 J}^{k_i}_{x_i}.
$$
We first study the cohomology $H^0_{\tilde{L}}$, $H^1_{\tilde{L}}$
of $\tilde{L}^{Y,D}_{\Sigma,f}$. There is a short exact sequence
\begin{equation}\label{proof-1}
0 \longrightarrow V\longrightarrow TY\longrightarrow
\pi^*TZ\longrightarrow 0,
\end{equation}
where $V$ is the vertical tangent bundle. It induces a short exact
sequence
\begin{equation}\label{proof-2}
0\longrightarrow f^*V\longrightarrow f^*TY\longrightarrow
f^*\pi^*TZ\longrightarrow 0.
\end{equation}

Choose a Hermitian metric and a unitary connection on $L$. It
induces a splitting of $(\ref{proof-1})$. We choose a metric of
$TY$ as the direct sum of the metric on $V$ and $\pi^*TZ$, where
the second one is induced from a metric on $Z$. The Levi-Civita
connection is a direct sum. Then $f^*V$ is a holomorphic line
bundle with respect to pullback of the Levi-Civita connection.

$(\ref{proof-2})$ induces a long exact sequence in cohomology
\begin{eqnarray*}
 0 & \longrightarrow & H^0(f^*V) \longrightarrow H^0(f^*TY)
 \longrightarrow H^0(f^*\pi^*TZ)\\
 & &\\
 & \longrightarrow & H^1(f^*V) \longrightarrow
 H^1(f^*TY)\longrightarrow H^1(f^*\pi^*TZ)\longrightarrow 0.
\end{eqnarray*}
It induces exact sequences
\begin{eqnarray*}
 &0 \longrightarrow H^0_{\tilde{L}}(f^*V)\longrightarrow
 H^0_{\tilde{L}}(f^*TY)\longrightarrow
 H^0_{\tilde{L}}(f^*\pi^*TZ),\\
&\\
& H^1_{\tilde{L}}(f^*V)\longrightarrow
H^1_{\tilde{L}}(f^*TY)\longrightarrow
H^1_{\tilde{L}}(f^*\pi^*TZ)\longrightarrow 0.
\end{eqnarray*}

Note that the normal bundle at the zero or infinity section is the
restriction of $V$. It is obvious that
$H^1_{\tilde{L}}(f^*\pi^*TZ) = H^1(f^*\pi^*TZ)$. An element of
$H^1_{\tilde{L}}(f^*V)$ with residue in $TZ$ must have zero
residue. Therefore,
\begin{equation}\label{proof-3}
H^1_{\tilde{L}}(f^*V) = H^1(\tilde{f}^*V),
\end{equation}
where $(\tilde{\Sigma}, \tilde{f})$ is obtained from $(\Sigma, f)$
by dropping the new marked points. For the same reason,
$$
   H^0_{\tilde{L}}(f^*V) = \{v\in H^0(f^*V)\mid
v(x_i) = 0\}.
$$

We claim that $H^1(\tilde{f}^*V) = 0$. Note that since $\Sigma$ is
a tree of ${\mathbb P}^1$'s, we see that $H^1(L)=0$ for any line
bundle $L$ on $\Sigma$ satisfying $\deg(L|_{\overline{\Sigma}})
\geq 0$ for any irreducible component $\overline{\Sigma}$ of
$\tilde{\Sigma}$.

Now we have
\begin{eqnarray*}
\deg (\tilde{f}^*V|_{\overline{\Sigma}}) &=&
\tilde{f}_*[\tilde{\Sigma}]\cdot c_1(V) =
\tilde{f}_*[\overline{\Sigma}]\cdot (\pi^*c_1(L)+2\xi)\\
& = & c_1(N)\cdot(\pi\circ \tilde{f})_*[\tilde{\Sigma}] +
2(\tilde{f}_*[\tilde{\Sigma}]\cdot \xi) \geq 0.
\end{eqnarray*}
Applying $L = \tilde{f}^*V$, we conclude that
$H^1_{\tilde{L}}(\tilde{f}^*V) = 0$. Next, we show that
\begin{equation}\label{proof-4}
  \oplus_i T^{k_i}_{x_i}: H^0_{\tilde{L}}(f^*TY)\longrightarrow
  \oplus_i{\mathcal T}^{k_i}_{x_i}
\end{equation}
is surjective. It is enough to show that the restriction to
$H^0_{\tilde{L}}(f^*V)$ is surjective. Consider the exact sequence
$$
  0\longrightarrow \tilde{f}^*V\otimes_i{\mathcal
O}(-k_ix_i) \longrightarrow \tilde{f}^*V\longrightarrow
\oplus_i\tilde{f}^*V_{k_ix_i}\longrightarrow 0.
$$

It induces a long exact sequence
$$
H^0(\tilde{f}^*V) \longrightarrow
\oplus_iH^0(\tilde{f}^*V_{k_ix_i}) \longrightarrow
H^1(\tilde{f}V\otimes_i{\mathcal O}(-k_ix_i)).
$$

Over each $\overline{\Sigma}$,
\begin{eqnarray*}
& &\deg \tilde{f}^*V\otimes_i{\mathcal
O}(-k_ix_i)\mid_{\overline{\Sigma}} =
\tilde{f}_*[\overline{\Sigma}]\cdot(\pi^*c_1(N) + 2\xi)\\
& & = c_1(N)\cdot (\pi\circ \tilde{f})_*[\overline{\Sigma}] +
2(\tilde{f}_*[\tilde{\Sigma}])\cdot \xi \geq 0.
\end{eqnarray*}

Hence, $H^1(\tilde{f}^*V\otimes_i{\mathcal O}(-k_ix_i)) =0$. This
implies that
$$
  H^0(\tilde{f}^*V)\longrightarrow \oplus H^0(\tilde{f}V_{k_ix_i})
$$
is surjective. Now we go back to $f$ and drop the constant term in
$H^0(\tilde{f}^*V_{k_ix_i})$. $(\ref{proof-4})$ becomes
\begin{equation}
   H^0_{\tilde{L}}(f^*V)\longrightarrow \oplus_i{\mathcal
   T}^{k_i}_{x_i},
\end{equation}
which is obviously surjective. By $(\ref{proof-3})$, we have
proved that $\mbox{Coker}(\tilde{L}^{Y,D}_{\Sigma,f}\times \sum
T^{k_i}_{x_i})$ is isomorphic to $H^1(f^*\pi^*TZ)$. Then, we argue
that $H^1(f^*\pi^*TZ)$ is isomorphic to $H^1(\pi (f)^*TZ)$. This
is obvious if $\pi (f)$ contracts an unstable component ${\mathbb
P}^1$, $\pi\circ f({\mathbb P}^1) =$ constant and ${\mathbb P}^1$
has one or two special points. Moreover, $\pi(\Sigma)$ is obtained
by contracting ${\mathbb P}^1$. Note that
$f^*\pi^*TZ\mid_{{\mathbb P}^1}$ is trivial.

The space of meromorphic $1$-forms on ${\mathbb P}^1$ with a
simple pole at one or two points is zero or $1$-dimensional. If
${\mathbb P}^1$ has only one special point, the residue at the
special point has to be zero. We can simply contract this
component and remove the pole at the other component which
${\mathbb P}^1$ is connected to. If ${\mathbb P}^1$ has two
special points, the residues at the two points have to be the
same. Then we can remove this component and the joint residue at
the two special points. Then we identify $H^1(f^*\pi^*TZ)$ and
$H^1(\pi (f)^*TZ)$.

Suppose that $(\Sigma,f)$ has more than one subfactor. Both
$\mbox{Coker}\,L^{Y,D}_{\Sigma,f}$ and
$\mbox{Coker}\,L^Z_{\pi(\Sigma,f)}$ are obtained by requiring the
residues at the new marked points to be opposite to each other.
Then our proof also extends to this case. Then we finish the proof
of Lemma \ref{lemma3-2}.

\end{proof}

Next, we continue the proof of  Proposition \ref{prop3-1} . Since
we have identified the obstruction spaces, we first choose a
stabilization term $\eta_i$ on $\overline{\mathcal
B}^Z_{\pi_*(A)}(0,k+l+q,J)$ to dominate the local obstruction
bundle generated by $\mbox{Coker}\,L^Z_{\pi(\Sigma,f)}$. Then, we
pull back $\eta_i$ over $\overline{\mathcal
B}_{\Gamma_0,T_k}(Y,D,J)$. By Lemma \ref{lemma3-2}, it dominates
$\mbox{Coker}\,L^{Y,D}$. This implies that $\pi$ induces a smooth
map on virtual neighborhood and a commutative diagram on
obstruction bundles
\begin{equation}\label{proof-5}
\begin{array}{ccccc}
  & &E_{Y,D}&\longrightarrow &E_Z\\
 & &\downarrow & & \downarrow\\
 \pi_{{\mathcal S}_e}& :& U^{Y,D}_{{\mathcal S}_e} & \longrightarrow &
 U^Z_{{\mathcal S}_e}.\end{array}
\end{equation}

Furthermore, the proper sections $S_{Y,D}$, $S_Z$ commutes with
the above diagram. $\pi_{{\mathcal S}_e}$ commutes with the
evaluation map for those $\beta_i$ classes. Choose a Thom form
$\Theta$ of $E_Z$. Its pullback is the Thom form on $E_{Y,D}$
(still denoted by $\Theta$).

It is clear that
$$
 \dim D_i = \dim U^{Y,D}_{{\mathcal S}_e} -2.
$$
By our construction, the $D_i$ intersect each other transversely.
Note that \begin{eqnarray*}
 \dim U^{Y,D}_{{\mathcal S}_e} & = & \mbox{rank}\, E_Z + 2(c^Y_1(A) + n-3 + l+q +k-\sum t_i).\\
 \dim U^Z_{{\mathcal S}_e} & = & \mbox{rank}\, E_Z + 2(c^Z_1(\pi_*(A)) + n - 1 - 3 + l+q+k).
\end{eqnarray*}
However,
$$
   c_1^Y(A) = c_1^Z(\pi_*(A)) + c_1(N)(\pi_*(A)) + 2\sum t_i = c_1^Z(\pi_*(A)) +
Z^*(A) +\sum t_i.
$$
Hence,
\begin{equation*}
\dim U^{Y,D}_{{\mathcal S}_e} - \dim U^Z_{{\mathcal S}_e} =
2(Z^*(A) +1) .
\end{equation*}
By definition, we have
\begin{eqnarray}\label{deg}
& & \deg \Theta + \sum_i\deg(\beta_i) + \deg\varpi + \sum_j\deg (\delta_j)\nonumber\\
& =& \dim U^{Y,D}_{{\mathcal S}_e} - 2d > \dim U^Z_{{\mathcal
S}_e},
\end{eqnarray}
where $\deg\varpi = \sum\deg\alpha_j$. Then, from (\ref{deg})and
$Z^*(A)\geq d$, we have
\begin{eqnarray*}
& &  (S_{Y,D})^*\Theta\prod
\psi_i^{d_i-1}ev^*_i\beta_i\wedge ev^*\varpi\wedge \prod ev^*_j\delta_j\\
&=& ((S_{Y,D})^*\Theta\prod \psi_i^{d_i-1}(\pi_{{\mathcal
S}_e})^*(ev^*_i\beta_i\wedge\prod ev^*\varpi\wedge \prod
ev^*_j\delta_j) = 0.
\end{eqnarray*}
In the last equality, we use $ev^*_i\beta_i\wedge\prod
ev^*\varpi\wedge \prod ev^*_j\delta_j=0$ on $U^Z_{{\mathcal
S}_e}$. Hence, the relative invariant is zero. This completes the
proof of
 Proposition \ref{prop3-1}.

\end{proof}

\begin{remark}
McDuff also proved the same result in the case without insertion
classes $\beta_i$ by a totally different method; see Lemma $1.7$
in \cite{M1}.
\end{remark}

\subsection{A nonvanishing theorem} When we apply the degeneration
formula, we often need to compute some special terms where the
degeneration graph lies completely on the side of the projective
bundle. In the previous subsection, we proved a vanishing theorem
for some relative invariants of $(Y,D)$. In this subsection, we
will consider the case where the relative invariants of $(Y,D)$
with empty relative insertion on $D$ are no longer zero and the
invariant of $Z$ will contribute in a nontrivial way.

Suppose that $A\in H_2(Z, {\mathbb Z})$. Denote by $i:
Z\longrightarrow Y$  the embedding of $Z$ into $Y$ as the zero
section. Consider  the relative invariant of $(Y,D)$
$$
\langle \tau_{i_1}(\beta_1[Z]),\tau_{i_2}(\beta_2[Z]),\cdots,
\tau_{i_k}(\beta_k[Z]),\varpi\mid\emptyset\rangle^{Y,D}_{0,A},
$$
where $\varpi$ consists of insertions of the form $\pi^*\alpha_1,
\cdots, \pi^*\alpha_l$. The dimension condition is
$$
2\sum(i_t +1) +\sum \deg\beta_t +\deg \varpi = 2(C_1^Y(A) + n-3 +
k+l).
$$
The dimension condition of the divisor invariant $\langle
\tau_{i_1}(\beta_1),\cdots,
\tau_{i_k}(\beta_k),i^*\varpi\rangle^Z_{0,A}$ is
$$
2(C_1^Z(A) +n-1-3 +k+l) = 2\sum i_t +\sum\deg\beta_t +\deg\varpi.
$$
Since $C_1^Y(A) = C_1^Z(A) + Z\cdot A$, so both invariants are
well-defined only when $k=Z\cdot A+1$.

\begin{theorem}\label{divisor-invariant}
Let $A\in H_2(Z, {\mathbb Z})$. Suppose that $k=Z\cdot A +1$ and
$C_1(L)(C)\geq 0$ for any holomorphic curve $C$ into $Z$. Then
\begin{eqnarray*}
& &\langle \varpi, \beta_1\cdot [Z],\beta_2\cdot [Z],\cdots,
 \beta_k\cdot [Z]\mid\emptyset\rangle^{Y,D}_{0,A}\\
& =&\langle \iota^*\varpi, \beta_1,\cdots,
 \beta_k\rangle^Z_{0,A},
\end{eqnarray*}
where $\varpi$ consists of insertions of the form $\pi^*\alpha_1,
\cdots, \pi^*\alpha_l$ and $\iota : Z\longrightarrow Y$ is the
embedding of $Z$ into $Y$ as the zero section.
\end{theorem}

\begin{proof}
Choose a Hermitian metric and a unitary connection on $L$ such
that they induce a splitting
$$
   0\longrightarrow V\longrightarrow TY \longrightarrow \pi^*TZ
   \longrightarrow 0,
$$
where $V$ is the vertical tangent bundle. We choose a metric of
$TY$ as the direct sum of a metric on $V$ and $\pi^*TZ$, where the
second one is induced from a metric on $Z$. The Levi-Civita
connection is a direct sum. Therefore, we may choose almost
complex structures $J_Z$ on $TZ$ and $J_V$  on $V$ such that we
may choose the direct sum $J_Z\oplus J_V$ as an almost complex
structure $J_Y$ on $TY$. It is easy to see that $\bar{\partial}$
commutes with $\pi$.

>From Lemma \ref{lemma3-2}, we know that the projection $\pi:
Y\longrightarrow Z$ induces a smooth map $\pi_{{\mathcal S}_e}:
U^{Y,D}_{{\mathcal S}_e}\longrightarrow U^Z_{{\mathcal S}_e}$ on
virtual neighborhoods and the obstruction bundle $E_{Y,D}$ over
$U^{Y,D}_{{\mathcal S}_e}$ is the pullback of the obstruction
bundle $E_Z$ over $U^Z_{{\mathcal S}_e}$. Therefore, by the
definition of Gromov-Witten invariants, we have
\begin{eqnarray}
& &\langle \varpi\beta_1\cdot [Z],\beta_2\cdot [Z],\cdots,
 \beta_k\cdot [Z]\mid\emptyset\rangle^{Y,D}_{0,A}\nonumber\\
& =&\deg (\pi_{{\mathcal S}_e})\langle i^*\varpi, \beta_1,\cdots,
 \beta_k\rangle^Z_{0,A}.
\end{eqnarray}

We claim that $\deg (\pi_{{\mathcal S}_e})=1$.

In fact, by the construction of virtual neighborhoods, we know
that for every generic element $({\mathbb P}^1, x_1,\cdots,
x_{k+l}, \tilde{f})\in U^Z_{{\mathcal S}_e}$, there is a section
$\nu$ of the obstruction bundle $E_Z$ such that $
    \bar{\partial}_{J_Z}\tilde{f} = \nu.
$

Suppose that a generic element $({\mathbb P}^1,x_1,\cdots,
x_{k+l}, f)\in U^{Y,D}_{{\mathcal S}_e}$ is a preimage of
$({\mathbb P}^1,x_1,\cdots, x_{k+l},\tilde{f})$ under
$\pi_{{\mathcal S}_e}$. That is, $f$ is a lifting of $\tilde{f}$
to $Y$ vanishing at the marked points $x_1,\cdots, x_k$.
Therefore, from the fact that $\bar{\partial}$ commutes with
$\pi$, we have that $({\mathbb P}^1,x_1,\cdots,x_{k+l}, f)$
satisfies
\begin{equation}\label{CR}
\bar{\partial}_{J_Y} f = \pi^*_{{\mathcal S}_e}\nu.
\end{equation}
If we choose a local coordinate $(z, s)$ on $Y$, where $s$ is the
Euclidean coordinate on the fiber ${\mathbb P}^1$, then locally we
may write $f = (\tilde{f}, f^V)$. Therefore (\ref{CR}) locally can
be written as
\begin{equation}\label{CR-1}
   \left\{\begin{array}{lcl}
    \bar{\partial}_{J_Z}\tilde{f} &=& \nu,\\
     & &\\
    \bar{\partial}_{J_V}f^V &=& 0
    \end{array}\right.
\end{equation}

Since $\bar{\partial}^2 = 0$ always holds on ${\mathbb P}^1$, it
follows from a well-known fact of complex geometry that $f^*L$ is
a holomorphic line bundle over ${\mathbb P}^1$. Moreover,
(\ref{CR-1}) shows that $f^V$ gives rise to a holomorphic section
of the bundle $f^*L$, up to ${\mathbb C}^*$, which vanishes at the
marked points $x_1,\cdots,x_k$. Since $\deg (f^*L) = Z\cdot A$,
therefore, from our assumption that $k=Z\cdot A +1$, we know that
$f^*L\otimes (-x_1-\cdots - x_k)$ has no nonzero holomorphic
sections. Therefore, $f^V\equiv 0$. This says that the only
preimage of a generic element $({\mathbb P}^1, x_1,\cdots,
x_{k+l}, \tilde{f})$ in $U^Z_{{\mathcal S}_e}$ is itself. This
implies $\deg (\pi_{{\mathcal S}_e}) = 1$. This proves the
theorem.

\end{proof}

\section{ A comparison theorem }

Let $X$ be a compact symplectic manifold and $Z\subset X$ be a
smooth symplectic submanifold of codimension $2$. $\iota :
Z\longrightarrow X$ is the inclusion map. The cohomological
push-forward
$$
\iota^! : H^*(Z, {\mathbb R}) \longrightarrow H^*(X,{\mathbb R})
$$
is determined by the pullback $\iota^*$ and Poincar\'e duality.

\begin{definition}
A symplectic divisor $Z$ is said to be {\em positive} if for some
tamed almost complex structure $J$, $C_1(N_Z)(A)>0$ for any $A$
represented by a non-trivial $J$-sphere in $Z$.
\end{definition}

This is a generalization of ample divisor from algebraic geometry.
Define
\begin{equation}\label{min-number}
  V := \min\{C_1(N_{Z|X})(A)>0\mid A\in H_2(Z,{\mathbb Z})\,\mbox{is a stably effective class}\}.
\end{equation}

In \cite{MP}, the authors point out that the relative
Gromov-Witten theory of $(X,Z)$ does not provide new invariants:
the relative Gromov-Witten theory of $(X,Z)$ is completely
determined by the absolute Gromov-Witten theory of $X$ and $Z$ in
principle. In this section, under some positivity assumptions on
the normal bundle of the divisor, we will give an explicit
relation between the absolute and relative Gromov-Witten
invariants, which we call as a comparison theorem. The main tool
of this section is the degeneration formula of Gromov-Witten
invariants. The central theorem of this section is

\begin{theorem}\label{comparison}Suppose that $Z$ is a positive
divisor and  $V\geq  l$.  Then for  $A\in H_2(X,{\mathbb Z})$,
$\alpha_i\in H^*(X,{\mathbb R})$, $ 1\leq i\leq \mu$, and
$\beta_j\in H^*(Z,{\mathbb R})$, $1\leq j\leq l$, we have
\begin{eqnarray*}
&&\langle \alpha_1, \cdots,  \alpha_\mu,
 \iota^!(\beta_1), \cdots,  \iota^!(\beta_l)
 \rangle ^X_A\\
&=&  \sum_{\mathcal T}\langle  \alpha_1,\cdots, \alpha_\mu\mid
{\mathcal T}\rangle_A^{X,Z},
\end{eqnarray*}
where the summation runs over all possible weighted partitions
${\mathcal T}= \{(1, \gamma_1)$, $\cdots, (1,\gamma_q),
(1,[Z]),\cdots, (1,[Z])\}$  where $\gamma_i$'s are the products of
some $\beta_j$ classes.
\end{theorem}

\begin{proof}
We perform the symplectic cutting along the boundary of a tubular
neighborhood of $Z$. Then we have $\overline{X}^- = X$,
$\overline{X}^+ = {\mathbb P}(N_{Z|X}\oplus {\mathbb C})$. Since
$\beta_i\in H^*(Z,{\mathbb R})$, we choose the support of
$\iota^!(\beta_i)$ near $Z$. Then, $\iota^!(\beta_i)^- = 0$,
$\iota^!(\beta_i)^+ = \iota^!(\beta_i)$. Here $\iota$ in the
second term is understood as the inclusion map of $Z$ via the zero
section into $Y={\mathbb P}(N_{Z|X}\oplus {\mathbb C})$. Up to a
rational multiple, each $\alpha_i$ is Poincar\'e dual to an
immersed submanifold $W_i$. We can perturb $W_i$ to be transverse
to $Z$. In a neighborhood of $Z$, $W_i$ is $\pi^{-1}(W_i\cap Z)$,
where $\pi: N_{Z|X} \longrightarrow Z$ is the projection. Clearly,
$\pi$ induces the projection  ${\mathbb P}(N_{Z|X}\oplus {\mathbb
C})\longrightarrow Z$, still denoted by $\pi$. The symplectic
cutting naturally decomposes $W_i$ into $W^-_i = W_i$, $W^+_i =
h^*(\alpha_i|_Z) = h^*\tilde{\alpha}_i$. In other words, we can
choose $\alpha_i^- = \alpha_i$, $\alpha^+_i = h^*(\alpha_i|_Z)$.

Now we apply the degeneration formula for invariants
$$
\langle \alpha_1,\cdots, \alpha_\mu,
 \iota^!(\beta_1),\cdots, \iota^!(\beta_l)
 \rangle_A^X
$$
and express it as a summation of products of relative invariants
of $(X,Z)$ and $(Y,D)$. Moreover, from the degeneration formula,
each summand $\Psi_C$ may consist of a product of relative
Gromov-Witten invariants with disconnected domain curves of both
$(X,Z)$ and $(Y,D)$.

On the side of  $Y$, there may be several disjoint components. Let
$A'$  be the total homology class. Then, from our assumption, we
have  $Z\cdot A' = Z\cdot A \geq d$. Suppose that we have a
nonzero summand $\Psi_C\not= 0$. We claim that each factor from
the relative Gromov-Witten invariants of $(Y,D)$ must be in the
form $\langle  \iota^!(\beta_{i_1}),\cdots, \iota^!(\beta_{i_t})|
(1,\gamma)\rangle^{Y,D}_F $  or $\langle\,|
(1,[pt])\rangle^{Y,D}_{F} $,where $F$ is the homology class of a
fiber of $Y$ and $\gamma$ is a basis element of $H^*(D,{\mathbb
R})$ such that
$\int_Z\beta_{i_1}\wedge\cdots\wedge\beta_{i_t}\wedge\gamma\not=0
$. Note that for these components, $Z^*(sF) =s$. From our
assumption $V\geq d$ and Proposition \ref{prop3-1}, the nonzero
factor of the relative Gromov-Witten invariants of $(Y,D)$ must be
the fiber class relative invariants. From Proposition
\ref{prop3-5}, we know that the nonzero factor must be of the form
$\langle  \iota^!(\beta_{i_1}),\cdots, \iota^!(\beta_{i_t})|
(1,\gamma)\rangle^{Y,D}_F $  or $\langle\,|
(1,[pt])\rangle^{Y,D}_{F} $. Moreover, if some $\beta_i = [pt]$,
then the marked point must be in a two-point component and the
nonzero relative invariant must be $\langle \iota^!([pt])|(1,
[Z])\rangle^{Y,D}_F =1$.

 Since there are no vanishing two-cycles in this case,
 we may write down the summation as follows.
\begin{eqnarray*}
& & \langle  \alpha_1, \cdots,
  \alpha_\mu,  \iota^!(\beta_1),\cdots,
  \iota^!(\beta_l)\rangle^X_A\\
& = & \sum_{\mathcal T}\langle  \alpha_1,\cdots, \alpha_\mu\mid
{\mathcal T}\rangle_A^{X,Z},
\end{eqnarray*}
 where the summation runs over all possible weighted partitions
${\mathcal T}= \{(1, \gamma_1)$, $\cdots, (1,\gamma_q),
(1,[Z]),\cdots, (1,[Z])\}$  where $\gamma_i'$ are the product of
some $\beta_j$ classes. We complete the proof of our comparison
theorem.

\end{proof}

\begin{cor}
Under the assumption of Theorem \ref{comparison}. If the product
of any two $\beta_j$ classes vanishes, then we have
$$
\langle \alpha_1,\cdots,\alpha_\mu,
\iota^!(\beta_1),\cdots,\iota^!(\beta_l)\rangle^X_A = \langle
\alpha_1, \cdots, \alpha_l\mid {\mathcal T}\rangle^{X,Z}_A,
$$
where ${\mathcal T} = \{(1,\beta_1),\cdots,
(1,\beta_l),(1,[Z]),\cdots,(1,[Z])\}$ is a weighted partition of
$Z\cdot A$.
\end{cor}

\section{Rationally connected symplectic divisors}

\subsection{Rationally connectedness in algebraic geometry} The basic reference for this
subsection is \cite{A}. We refer to \cite{C, D, K, KMM1, KMM2, V}
for more details.

Let us recall the notion of rational connectedness in algebraic
geometry.
\begin{definition}
Let $X$ be a smooth complex projective variety of positive
dimension. We say that $X$ is rationally connected if one of  the
following equivalent conditions holds.
\begin{enumerate}
\item Any two points of $X$ can be connected by a rational curve
(called as rationally connected).

\item Two general points of $X$ can be  connected by a chain of
rational curves (called as rationally chain-connected).

\item Any finite set of points in $X$ can be connected by a
rational curve.

\item Two general points of $X$ can be connected by a very free
rational curve. Here we say that a rational curve $C\subset X$ is
a very free curve if there is a surjective morphism $f: {\mathbb
P}^1\longrightarrow C$ such that
$$
  f^*X \cong \oplus_{i=1}^{\dim X}{\mathcal O}_{{\mathbb
  P}^1}(a_i), \,\,\,\,\, \mbox{with all} \,\, a_i \geq 1.
$$

\end{enumerate}

\end{definition}

Next let us look at some properties of rationally connected
varieties.
\begin{prop} The following properties of rationally connected manifolds hold:
\begin{enumerate}
\item Rationally connectedness is a birational invariant.

\item Rationally connectedness is invariant under smooth
deformation.

\item If $X$ is rationally connected, then
$H^0(X,(\Omega^1_X)^{\otimes m})=0$ for every $m\geq 1$.

\item Fano varieties (i.e., smooth complex projective varieties
$X$ for which $-K_X$ is ample) are rationally connected. In
particular, smooth hypersurfaces  of degree $d$ in ${\mathbb P}^n$
are rationally connected for $d\leq n$.

\item Rationally connected varieties are well behaved under
fibration, i.e.:  Let $X$ be s smooth complex projective variety.
Assume that there exists a surjective morphism $f:X\longrightarrow
Y$  with $Y$ and the general fiber of $f$ rationally connected.
Then $X$ is rationally connected.

\end{enumerate}
\end{prop}
 An important theorem connecting birational geometry to
 Gromov-Witten theory is the result of Koll\'ar and Ruan
 \cite {K}, \cite{R1}: a uniruled projective manifold has a nonzero
genus zero GW-invariant with a point insertion. A longstanding
problem in Gromov-Witten theory is that a similar result with two
point insertions should also hold for rationally connected
projective manifolds.

\subsection{Rationally connected symplectic divisors} In
this subsection, we want to apply our comparison theorem to study
the $k$-point rationally connectedness properties. We will show
that a symplectic manifold is  $k$-point strongly rationally
connected if it contains a k-point strongly rationally connected
symplecitc divisor with sufficiently positive normal bundle. Our
main tool is the degeneration formula of Gromov-Witten invariants
and our comparison Theorem \ref{comparison}.

Before we state our main theorem, we want first to define the
notion of $k$-point rational connectedness.

\begin{definition}\label{rs-class}
Let $A\in H_2(X, {\mathbb Z})$ be a nonzero class. $A$ is said to
be a $k$-point rationally connected class if there is a nonzero
Gromov-Witten invariant
\begin{equation}\label{rcdef}
  \langle \tau_{d_1}[pt],\cdots,\tau_{d_k}[pt],\tau_{d_{k+1}}\alpha_{k+1},\cdots,
  \tau_{d_l}\alpha_l\rangle^X_A,
\end{equation}
where $\alpha_i\in H^*(X,{\mathbb R})$ and $d_1,\cdots, d_l$ are
non-negative integers. We call a class $A\in H_2(X, {\mathbb Z})$
a  $k$-point strongly rationally connected class if $d_i =0$,
$1\leq i\leq l$, in (\ref{rcdef}).
\end{definition}

\begin{definition}\label{rc}
 $X$ is said to be (symplectic) $k$-point (strongly ) rationally connected
 if there is a $k$-point (strongly) rationally connected class. We simply call a $2$-point (strongly)
 rationally
 connected symplectic manifold as (strongly) rationally connected symplectic
 manifold.
\end{definition}

\begin{remark}From the definition of uniruledness of \cite{HLR},
 a $1$-point rationally connected symplectic manifold is equivalent
 to a
 uniruled symplectic manifold.  From the definitions, we know that a $k$-point (strongly)
 rational connected symplectic manifold must be uniruled.
\end{remark}

\begin{remark} It is possible that $k$-point rational
connectedness is equivalent to $k$-point strongly rational
connectedness. We do not know how to prove this.
\end{remark}

\begin{example}
It is well-known that for any positive integer $k$, the projective
space ${\mathbb P}^n$ is $k$-point strongly rationally connected.
\end{example}

\begin{example}
Let $G(k,n)$ be the Grassmannian manifold of $k$-planes in
${\mathbb C}^n$. It is well-known that the classical cohomology of
$Gr(k,n)$ has a basis of Schubert classes $\sigma_\lambda$, as
$\lambda$ varies over partitions whose Young diagram fits in a $k$
by $n-k$ rectangle. The (complex) codimension of $\sigma_\lambda$
is $|\lambda|= \sum \lambda_i$, the number of boxes in the Young
diagram.  The quantum cohomology of the Grassmannian is a free
module over the polynomial ring ${\mathbb Z}[q]$, with a basis of
Schubert classes; the variable $q$ has (complex) degree $n$. The
quantum product $\sigma_\lambda \star \sigma_\mu$ is a finite sum
of terms $q^d\sigma_\nu$, the sum over $d\geq 0$ and $|\nu| =
|\lambda| + |\mu| - dn$, each occurring with a nonnegative
coefficient (a Gromov-Witten invariant). Denote by
$\rho=\sigma_{((n-k)^k)}$ the class of a point. In \cite{BCF}, the
authors proved that $\sigma_\rho \star \sigma_\rho =
q^k\sigma_{((n-2k)^k)}$ if $k\leq n-k$, and $\sigma_\rho \star
\sigma_\rho = q^{n-k}\sigma_{((n-k)^{2k-n})}$ if $n-k \leq k$.
This means that the Grassmannian $Gr(k,n)$ is symplectic
rationally connected.
\end{example}

\begin{example}
   For any integer $d\geq 0$, Consider the Grassmannian $G(d,2d)$.
   Buch-Kresch-Tamvakis \cite{BKT} proved that for three points
   $U, V, W\in G(d,2d)$ pairwise in general position, there is a
   unique morphism $f: {\mathbb P}^1\longrightarrow G(d,2d)$ of
   degree $d$ such that $f(0)= U$, $f(1)=V$ and $f(\infty)=W$.
   This implies that the Gromov-Witten invariant $\langle [pt],
   [pt],[pt]\rangle^{G(d,2d)}_d =1$. Therefore, $G(d,2d)$ are
   $3$-point strongly rationally connected.
\end{example}

\begin{example}
Let ${\mathbb H} = \mbox{Hilb}^2({\mathbb P}^1\times {\mathbb
P}^1)$ the Hilbert scheme of points on ${\mathbb P}^1\times
{\mathbb P}^1$. In \cite{P}, the author gave a ${\mathbb Q}$-basis
for $A*({\mathbb H})$ as follows: $T_0 = [{\mathbb H}]$, $T_1$,
$T_2$, $T_3$, $T_4$, $T_5=T_1T_2$, $T_6= T_1^2$, $T_7=T_2^2$,
$T_8=T_1T_3$, $T_9=T_2T_3$, $T_{10}=C_2 +F$, $T_{11}=C_1+F$,
$T_{12}=C_1+C_2+F$ and $T_{13}$ the class of a point. The author
also computed the quantum product $T_4 \star T_4 = T_{13} +
2q_1q_2q_3^2 T_0$ and $T_4\star T_4\star T_4=2q_1q_2q_3^2T_4$. So
we have $T_{13}\star T_{13} = -2q_1q_2q_3^2T_{13}\not= 0$. This
implies that ${\mathbb H}$ is symplectic rationally connected.
>From the computation of \cite{G}, it is easy to know that
$\mbox{Hilb}^2({\mathbb P}^2)$ also is symplectic rationally
connected.
\end{example}

Since Gromov-Witten invariants are invariant under smooth
symplectic deformations, $k$-point rational connectedness is
invariant under smooth symplectic deformations. It is not yet
known whether a projective rationally connected manifold is
symplectic rationally connected. Moreover, so far, we could not
show that this notion is invariant under symplectic birational
cobordisms defined in \cite{HLR}. We will leave this for future
research. However, we would like to mention some partial results
along this direction. From the blowup formula of Gromov-Witten
invariants in \cite{H1,H2,H3,HZ,La}, we have

\begin{prop}\label{blowup}
Suppose that $X$ is a $k$-point strongly rationally connected
symplectic manifold. Let $\tilde{X}$ be the blowup of $X$ along a
finite number of points or some special submanifolds with convex
normal bundles (see, \cite{H1,La}). Then $\tilde{X}$ is $k$-point
strongly rationally connected.
\end{prop}

In 1991, McDuff \cite{M3} first observed that a semi-positive
symplectic $4$-manifold, which contains a submanifold $P$
symplectomorphic to ${\mathbb P}^1$ whose normal Chern number is
non-negative, must be uniruled. In \cite{LtjR}, the authors
generalize McDuff's result to more general situations. More
importantly, they gave a rather general {\it from divisor to
ambient space} inductive construction of uniruled symplectic
manifolds. In this subsection, we will generalize their inductive
construction to the case of rationally connected symplectic
manifolds.

Suppose that $X$ is a compact symplectic manifold and $Z\subset X$
is a symplectic submanifold of codimension $2$. Denote by
$N_{Z|X}$ the normal bundle of $Z$ in $X$. Denote by $\iota
:Z\subset X$ the inclusion of $Z$ into $X$. Let $V$ be the minimal
 normal Chern number defined in (\ref{min-number}).  We call a class $A\in H_2(Z,
{\mathbb R})$ a {\bf  minimal class} if $Z\cdot A =V$.

\begin{theorem}\label{rcsd}
Suppose that $X$ is a compact symplectic manifold and $Z\subset X$
is a symplectic submanifold of codimension $2$.  If  $Z$ is
k-point strongly rationally connected and $A\in H_2(Z,{\mathbb
Z})$ is a minimal class such that
\begin{equation}\label{rcsd-1}
  \langle \iota^*\alpha_1,\cdots, \iota^*\alpha_l,[pt],\cdots,[pt], \beta_{k+1},\cdots,\beta_r \rangle^Z_A \not=0
\end{equation}
for some $r\leq  V+1$, $\beta_i\in H^*(Z,{\mathbb R})$ and
$\alpha_j\in H^*(X,{\mathbb R})$, then $X$ is k-point strongly
rationally connected.  In particular, if $\iota : Z\rightarrow X$
is homologically injective, then $X$ is $k$-point strongly
symplectic rational connected if $k\leq V+1$.

\end{theorem}

\begin{proof}
Since we can always increase the number of $Z$-insertions by
adding divisor insertions in (\ref{rcsd-1}), therefore, without
loss of generality, we may assume that $r=Z\cdot A +1$. Consider
the following Gromov-Witten invariant of $X$:
\begin{equation}\label{x-invariant}
\langle \alpha_1,\cdots,\alpha_l,[pt] ,\cdots, [pt]
,\beta_{k+1}\cdot [Z],\cdots,\beta_r\cdot [Z]\rangle^X_A.
\end{equation}

If the invariant (\ref{x-invariant}) is nonzero, then we are done.
So in the following we assume that the invariant
(\ref{x-invariant}) equals zero.

To find a nonzero Gromov-Witten invariant of $X$ with at least $k$
point insertions, we first apply the degeneration formula to the
invariant (\ref{x-invariant}) to obtain a nonzero relative
Gromov-Witten invariant of $(X,Z)$ with at least $k$ point
insertions, then use our comparison theorem to obtain a nonzero
Gromov-Witten invariant of $X$.

We perform the symplectic cutting along the boundary of a tubular
neighborhood of $Z$. Then we have $\overline{X}^- = X$,
$\overline{X}^+ =Y= {\mathbb P}(N_{Z|X}\oplus {\mathbb C})$. Since
$\beta_i\in H^*(Z,{\mathbb R})$, we choose the support of
$[pt]\cdot Z$, $\beta_i\cdot Z$ near $Z$. Then,$([pt] )^- =0$,
$(\beta_i\cdot Z)^- = 0$, $([pt] )^+ = [pt]$, $(\beta_i\cdot Z)^+
= \iota^!(\beta_i)$. Here $\iota$ in the second term is understood
as the inclusion map of $Z$ via the zero section into $Y={\mathbb
P}(N_{Z|X}\oplus {\mathbb C})$. Up to a rational multiple, each
$\alpha_i$ is Poincar\'e dual to an immersed submanifold $W_i$. We
can perturb $W_i$ to be transverse to $Z$. In a neighborhood of
$Z$, $W_i$ is $\pi^{-1}(W_i\cap Z)$, where $\pi: N_{Z|X}
\longrightarrow Z$ is the projection. Clearly, $\pi$ induces the
projection ${\mathbb P}(N_{Z|X}\oplus {\mathbb C})\longrightarrow
Z$, still denoted by $\pi$. The symplectic cutting naturally
decomposes $W_i$ into $W^-_i = W_i$, $W^+_i = h^*(\alpha_i|_Z) =
h^*\tilde{\alpha}_i$. In other words, we can choose $\alpha_i^- =
\alpha_i$, $\alpha^+_i = h^*(\alpha_i|_Z)$.

Now we apply the degeneration formula for the  invariant
$$
\langle \alpha_1,\cdots,\alpha_l,[pt] ,\cdots, [pt]
,\beta_{k+1}\cdot [Z],\cdots,\beta_r\cdot [Z] \rangle^X_A.
$$
and express it as a summation of products of relative invariants
of $(X,Z)$ and $(Y,D)$. Moreover, from the degeneration formula,
each summand $\Psi_C$ may consist of a product of relative
Gromov-Witten invariants with disconnected domain curves of both
$(X,Z)$ and $(Y,D)$.

On the side of $Y$, there may be several disjoint components.   We
claim that every component on the side of $Y$ is a multiple of the
fiber class. Suppose that there is a component on the side of $Y$
which has a homology class $A' + \mu F$ where $A'\in
H_2(Z,{\mathbb Z})$ and $F$ is the fiber class. Then we have
$Z\cdot ( A'+\mu F) = Z\cdot A' +\mu$. Since every component must
intersect the infinity section of $Y$,  $\mu >0$. Therefore, we
have $Z\cdot (A'+\mu F)\geq V + 1$. From Proposition
\ref{prop3-1}, we know that the contribution of this component to
the corresponding relative Gromov-Witten invariants of $(Y,D)$
must be zero. So the corresponding summand in the degeneration
formula must be zero. From Proposition \ref{prop3-5} and the same
argument as in the proof of Theorem \ref{comparison} , we have
\begin{eqnarray*}
 0 &=& \langle \alpha_1,\cdots,\alpha_l, [pt] ,\cdots, [pt] ,\beta_{k+1}\cdot
[Z],\cdots,\beta_r\cdot [Z]
 \rangle^X_A\\
&=& \sum_\mu C_\mu \langle \alpha_1,\cdots,\alpha_l \mid \mu
 \rangle^{X,Z}_A\\
& & + \langle i^*\alpha_1,\cdots,i^*\alpha_l,[pt],\cdots, [pt], \beta_{k+1}\cdot [Z],\cdots, \beta_r\cdot [Z]\mid \emptyset \rangle^{Y,D}_A\\
& = & \sum_\mu C_\mu \langle \alpha_1,\cdots,\alpha_l \mid \mu
 \rangle^{X,Z}_A \\
& &+ \langle i^*\alpha_1,\cdots,i^*\alpha_l,[pt] ,\cdots, [pt] ,
\beta_{k+1} ,\cdots, \beta_r \rangle^Z_A,
\end{eqnarray*}
where we used Theorem \ref{divisor-invariant} in the last equality
and the summation runs over the possible partitions $\mu=\{(1,
[pt]), \cdots, (1,[pt]), (1, \gamma_1),\cdots, (1,\gamma_q),
(1,[Z])$, $\cdots,(1,[Z])\}$ where $\gamma_i$ are the product of
some $\beta_i$ classes.. From our assumption (\ref{rcsd-1}), we
have
\begin{equation}\label{nonzero-sum}
\sum_\mu C_\mu \langle \alpha_1,\cdots,\alpha_l \mid \mu
 \rangle^{X,Z}_A
\not= 0.
\end{equation}

Denote by $\langle \alpha_1,\cdots,\alpha_l\mid
\mu_0\rangle^{X,Z}_A$ the minimal nonzero relative invariant in
the summand (\ref{nonzero-sum})in the sense of Definition
\ref{order}. Write $\mu_0=\{(1, [pt]),\cdots$, $(1,[pt]),
(1,\gamma_1),\cdots,(1,\gamma_q),(1,[Z]),\cdots, (1,[Z])\}$. Then
it is easy know that the product of any two $\gamma_i$ and
$\gamma_j$ vanishes.

Now we consider the following absolute Gromov-Witten invariant
$$
\langle \alpha_1,\cdots,\alpha_l, [pt],\cdots,[pt],\gamma_1\cdot
[Z],\cdots,\gamma_l\cdot [Z],[Z],\cdots, [Z]\rangle^X_A.
$$
Applying the degeneration formula and always distribute the point
insertions to the side of $(Y,D)$. Therefore, from Theorem
\ref{comparison}, we have
\begin{eqnarray*}
& &\langle \alpha_1,\cdots,\alpha_l, [pt],\cdots,[pt],
\gamma_1\cdot
[Z], \cdots, \gamma_q\cdot [Z]\rangle^X_A\\
&=& \langle \alpha_1,\cdots,\alpha_l\mid \mu_0\rangle^{X,Z}_A
\not= 0.
\end{eqnarray*}
This implies that $X$ is $k$-point strongly rationally connected.
This proves our theorem.

\end{proof}

It is well-known that ${\mathbb P}^{n-1}$ is strongly rationally
connected. Therefore, from Theorem \ref{rcsd}, we have
\begin{cor}
Let $(X,\omega)$ be a compact $2n$-dimensional symplectic manifold
which contains a submanifold $P$ symplectomorphic to ${\mathbb
P}^{n-1}$ whose normal Chern number  $m\geq 2$. Then $X$ is
strongly rationally connected.

\end{cor}


\begin{thebibliography}{999}
\bibitem[A]{A}C. Araujo, Rationally connected varities,
arXiv:math.AG/0503305
\bibitem[BCF]{BCF}A. Bertram, I. Ciocan-Fontanine, W. Fulton,
Quantum multiplication of Schur Polynomials, J. Alg., 219(1999),
728-746.
\bibitem[BKT]{BKT}A. S. Buch, A. Kresch, H. Tamvakis,
Gromov-Witten invariants on Grassmannians, J. Amer. Math. Soc.,
16(4)(2003), 901-915.
\bibitem[C]{C}F. Campana, Connexit\'e tationanelle des vari\'et\'es
de Fano, Ann. Sci. \'Ecole Norm. Sup. 25(4)(1992), no.5, 539-545.
\bibitem[D]{D}O. Debarre, Higher-dimensional algebraic geometry,
Universitext, Springer-Verlag, New York, 2001.
\bibitem[G]{G}T. Graber, Enumerative geometry of hyperelliptic
plane curves, J. Alg. Geom., 10(2001), 725-755.
\bibitem[H1]{H1}J. Hu, Gromov-Witten invariants of blow-ups along
points and curves, Math. Z. 233(2000), 709-739.
\bibitem[H2]{H2}J. Hu, Gromov-Witten invariants of blow-ups along surfaces,
Compositio Math., 125(2001),345-352.
\bibitem[H3]{H3}J. Hu, Gromov-Witten invariants of blowups along
semi-positive submanifold, Adv. in Math. Research, vol. 3(2003),
103-109.
\bibitem[HLR]{HLR}J. Hu, T.-J. Li, Y, Ruan, Birational cobordism
invariance of uniruled symplectic manifolds, math.SG/0611592, to
appear in Invent. Math.
\bibitem[HZ]{HZ}J. Hu, H. Zhang, Elliptic Gromov-Witten invariants
of blowups along surfaces, Int. J. Math. Math. Sci, 2005:1(2005),
81-90.
\bibitem[IP]{IP}E. Ionel, T. Parker, Relative Gromov-Witten
invariants, Ann. of Math., 157(2)(2003), 45-96.
\bibitem[K]{K}J. Koll\'ar, Rationally curves on algebraic
varieties, Ergebnisse der Mathematik und ihrer Grenzgebiete, vol.
32, Springer-Verlag, Berlin, 1996.
\bibitem[KMM1]{KMM1}J. Koll\'ar, Y. Miyaoka, S. Mori, Rational
connectedness and boundedness of Fano manifolds, J. Differential
Geom. 36(3)(1992), 765-779.
\bibitem[KMM2]{KMM2}J. Koll\'ar, Y. Miyaoka, S. Mori, Rationally
connected varieties, J. Algebraic Geom. 1(3)(1992),429-448.
\bibitem[La]{La}H. Lai, Gromov-Witten invariants of blowups along
submanifolds with convex normal bundles, arXiv:math.AG/0710.3968.
\bibitem[L]{L}E. Lerman, Symplectic cuts, Math. Research Lett.
2(1995), 247-258.
\bibitem[Li1]{Li1}J. Li, Stable morphisms to singular schemes and
relative stable morphisms, J. Diff. Geom. 57(2001),509-578.
\bibitem[Li2]{Li2}J. Li, Relative Gromov-Witten invariants and a
degeneration formula of Gromov-Witten invariants, J. Diff. Geom.
60(2002), 199-293
\bibitem[LR]{LR} A. Li, Y. Ruan, Symplectic surgery and Gromov-Witten invariants
      of Calabi-Yau 3-folds, Invent. Math. 145(2001), 151-218.
\bibitem[LtjR]{LtjR}T.-J. Li, Y. Ruan, Uniruled symplectic
divisors, arXiv:math.SG/0711.4254
\bibitem[MP]{MP} D. Maulik, R. Pandharipande, A topological view of
Gromov-Witten theory, Topology, 45(5)(2006), 887-918.
\bibitem[M1]{M1}D. McDuff, Comparing absolute and relative
Gromov-Witten invariants, preprint.
\bibitem[M2]{M2}D. McDuff, Hamiltonian $S^1$-manifolds are
uniruled, preprint, 2007
\bibitem[M3]{M3}D. McDuff, Symplectic manifolds with contact type
boundaries, Invent. Math. 103(1991), 651-671.
\bibitem[OP]{OP}A. Okounkov, R. Pandharipande, Gromov-Witten
theory, Hurwitz numbers, and  completed cycles, Ann. of Math.,
163(2)(2006), 517-560.
\bibitem[P]{P}D. Pontoni, Quantum cohomology of $\mbox{Hilb}^2({\mathbb P}^1\times {\mathbb
P}^1)$ and enumerative applications, Trans. Amer. Math. Soc.,
359(2007), 5419-5448.
\bibitem[Q]{Q}Z. Qin, private communication.
\bibitem[R1]{R1}Y. Ruan, Virtual neighborhoods and pseudoholomorphic
             curves, Turkish J. Math., 23(1999), 161-231.
\bibitem[R2]{R2} Y. Ruan, Surgery, Quantum cohomology and birational geometry,
Northern California Symplectic Geometry Seminar, 183-198, Amer.
Amth. Soc. Transl. Ser.2, 196. Amer. Math. Soc., Providence, RI,
1999.
\bibitem[R3]{R3}Y. Ruan, Topological sigma model and Donaldson
type invariants in Gromov theory, Duke Math. J.
83(2)(1996),461-500.
\bibitem[S1]{S1} N. Siebert, Gromov-Witten invariants for general symplectic
manifolds,  New trends in algebraic geometry(Warwick, 1996),
       375-424, Cambrdge Uni. Press, 1999.
\bibitem[V]{V} C. Voisin, Rationally connected $3$-folds and
symplectic geometry, arXiv:0801.1396.

\end{thebibliography}
\end{document}